\documentclass[12pt]{amsart}

\usepackage{amssymb}
\usepackage{commath}
\usepackage{verbatim}
\usepackage{thmtools}
\usepackage{subcaption}
\usepackage{geometry}
\usepackage{rotating}
\usepackage{tikz}
\usepackage{graphicx}
\usepackage{mathrsfs}

\DeclareCaptionLabelFormat{custom}
{%
      \textbf{Fig. #2}
}
\DeclareCaptionLabelSeparator{custom}{\,\,}
\usetikzlibrary{arrows}
\tikzstyle{vertex}=[circle, draw, inner sep=0pt, minimum size=4pt]
\newcommand{\vertex}{\node[vertex]}

\definecolor{pur}{RGB}{186,146,162}

\theoremstyle{definition}%
\newtheorem{theorem}{Theorem}[section]
\newtheorem{proposition}[theorem]{Proposition}%
\newtheorem{lemma}[theorem]{Lemma}%
\newtheorem{definition}[theorem]{Definition}%
\newtheorem{corollary}[theorem]{Corollary}%
\newtheorem{example}[theorem]{Example}%
\newtheorem{remark}[theorem]{Remark}%

\newcommand\be{\mathbf{e}}
\newcommand\bp{\mathbf{p}}
\newcommand\bq{\mathbf{q}}

\newcommand\bu{\mathbf{u}}

\newcommand\bx{\mathbf{x}}

\newcommand\bbR{\mathbb{R}}
\newcommand\bbZ{\mathbb{Z}}

\newcommand\calA{\mathcal{A}}

\newcommand\calE{\mathcal{E}}
\newcommand\calF{\mathcal{F}}

\newcommand\calP{\mathcal{P}}
\newcommand\calQ{\mathcal{Q}}
\newcommand\calR{\mathcal{R}}
\newcommand\calS{\mathcal{S}}
\newcommand\calT{\mathcal{T}}
\newcommand\calU{\mathcal{U}}

\newcommand\scrT{\mathscr{T}}

\newcommand\inv{^{-1}}

\DeclareMathOperator{\Cat}{Cat}
\DeclareMathOperator{\vol}{vol}

\usepackage[sorting=none,backend=bibtex]{biblatex}
\title{A subdivision algebra for a product of two simplices via flow polytopes}
\author[von Bell]{\normalsize{Matias von Bell$^{1}$
\footnote{$^1$Email: matias.vonbell@gmail.com
}}}
\address[von Bell]{Institute of Geometry, Graz University of Technology, Kopernikusgasse 24, Graz, A-8010, Austria}

\begin{document}

\begingroup
\let\MakeUppercase\relax 
\maketitle

\endgroup

\parskip=5pt

\begin{abstract}
For a lattice path $\nu$ from the origin to a point $(a,b)$ using steps $E=(1,0)$ and $N=(0,1)$, we construct an associated flow polytope $\calF_{\widehat{G}_B(\nu)}$ arising from an acyclic graph where bidirectional edges are permitted.
We show that the flow polytope $\calF_{\widehat{G}_B(\nu)}$ admits a subdivision dual to a $w$-simplex, where $w$ is the number of valleys in the path $\overline{\nu} = E\nu N$. 
Refinements of this subdivision can be obtained by reductions of a polynomial $P_\nu$ in a generalization of M\'esz\'aros' subdivision algebra for acyclic root polytopes where negative roots are allowed. 
Via an integral equivalence between $\calF_{\widehat{G}_B(\nu)}$ and the product of simplices $\Delta_a\times \Delta_b$, we thereby obtain a subdivision algebra for a product of two simplices. 
As a special case, we give a reduction order for reducing $P_\nu$ that yields the cyclic $\nu$-Tamari complex of Ceballos, Padrol, and Sarmiento.

\vspace{0.4cm} 
\vspace{0.2cm}


\end{abstract}

\thispagestyle{empty}

\section{Introduction} 
\label{sec1}

Given a standard $a$-simplex $\Delta_a$ and a standard $b$-simplex $\Delta_b$, their cartesian product is the $(a+b)$-dimensional polytope $$\Delta_a\times \Delta_b = \mathrm{conv}\{(\be_i,\be_j) \mid \be_i \in \Delta_a \text{ and } \be_j \in \Delta_b \}.$$  

The triangulations of $\Delta_a \times \Delta_b$ have been extensively studied. They serve as building blocks in the study of triangulations of more complicated polytopes \cite{Hai91,OS03}, and play an important role in understanding triangulations more generally \cite{DeL96,San00}.
They have also garnered interest from a variety of perspectives, having connections to algebraic geometry (Schubert calculus \cite{AB07}, Segre varieties \cite{CHT06}, Gr\"obner bases \cite{Stu96}), tropical geometry (tropical convexity \cite{DS03}, tropical hyperplane arrangements and oriented matroids \cite{AD09,CPS19}), and optimization (dual transportation polytopes \cite{DKOS09}).

In this article, we study $\Delta_a \times \Delta_b$ from the perspective of flow polytopes, which are a rich family of polytopes arising from placing flows through directed acyclic graphs. 
Flow polytopes have close connections to many areas including representation theory \cite{BV08}, diagonal harmonics \cite{MMR17}, Schubert polynomials \cite{MS17}, and toric geometry \cite{H03}. 
The product of simplices $\Delta_a\times \Delta_b$ is the flow polytope $\calF_{G_{a,b}}$, where the underlying graph $G_{a,b}$ has vertex set $\{s,1,t\}$, along with $a+1$ copies of the edge $(s,1)$ and $b+1$ copies of the edge $(1,t)$. 
The vertices $s$ and $t$ are respectively the source and sink of the graph. 
For example, the flow polytope of the graph on the left in Figure~\ref{fig:prodOfSimplices} is $\Delta_4\times \Delta_3$.
At first glance, the flow polytope perspective does not seem to yield anything new in regards to triangulations of $\Delta_a \times \Delta_b$. 
In particular, the Danilov--Karzanov--Koshevoy triangulations \cite{DKK12} (or equivalently Stanley--Postnikov triangulations \cite{MMS19}) of $\calF_{G_{a,b}}$ are the staircase triangulations of $\Delta_a \times \Delta_b$, which are well understood (see \cite[Section 6.2]{DRS10}).

To obtain more triangulations, we slightly expand the typical definition of a flow polytope in Section~\ref{sec:2}. 
The more general definition allows negative flows and permits the underlying graph $G$ to have bidirectional edges, as depicted in the graph on the right of Figure~\ref{fig:prodOfSimplices}.
This allows us to generate a class of $\binom{a+b}{a}$ flow polytopes in Section~\ref{sec:3} that are integrally equivalent to $\Delta_a\times \Delta_b$.
In particular, for any lattice path $\nu$ from the origin to a point $(a,b)$ in the first quadrant using steps $N=(0,1)$ and $E=(1,0)$, we construct a graph $\widehat{G}_B(\nu)$ and show that its flow polytope $\calF_{\widehat{G}_B(\nu)}$ is integrally equivalent to $\Delta_a\times \Delta_b$ (Lemma~\ref{lem:intEq}). 
Therefore triangulations of $\calF_{\widehat{G}_B(\nu)}$ give triangulation of $\Delta_a\times \Delta_b$.
Furthermore, we show in Theorem~\ref{thm:subdivision} that the lattice path $\nu$ induces a subdivision of $\calF_{\widehat{G}_B(\nu)}$ whose inner faces are integrally equivalent to flow polytopes. 
The dual to this induced subdivision is a $w$-simplex, where $w$ is the number of valleys (consecutive $EN$ pairs) in the path $\overline{\nu}:=E\nu N$. We therefore refer to it as the {\em simplex-subdivision} of $\calF_{\widehat{G}_B(\nu)}$. 

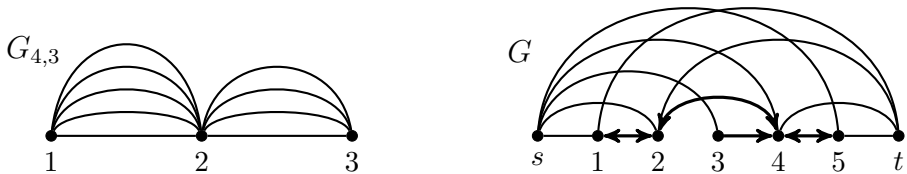
\begin{figure}
\begin{center}
\begin{tikzpicture}[scale=0.8]
\begin{scope}[xshift=230, yshift=0, scale=1]
	\vertex[fill,label=below:\small{$s$}](a1) at (1,0) {};
	\vertex[fill,label=below:\small{$1$}](a2) at (2,0) {};
	\vertex[fill,label=below:\small{$2$}](a3) at (3,0) {};
	\vertex[fill,label=below:\small{$3$}](a4) at (4,0) {};
	\vertex[fill,label=below:\small{$4$}](a5) at (5,0) {};
	\vertex[fill,label=below:\small{$5$}](a6) at (6,0) {};
	\vertex[fill,label=below:\small{$t$}](a7) at (7,0) {};

	\draw[>=stealth', thick] (a1)--(a2);
    \draw[>=stealth', <->, very thick] (a2)--(a3);	
	\draw[>=stealth', <->, very thick] (a3) .. controls (3.25,0.8) and (4.75,0.8) .. (a5);
	\draw[>=stealth', <->, very thick] (a5)--(a6);
	\draw[>=stealth', ->, very thick] (a4)--(a5);
	\draw[>=stealth', thick] (a6)--(a7);
	\draw[>=stealth', thick] (a1) .. controls (1.25, .7) and (2.75, .7) .. (a3);
	\draw[>=stealth', thick] (a1) .. controls (1.25, 1.4) and (3.75, 1.4) .. (a4);
	\draw[>=stealth', thick] (a1) .. controls (1.25, 2.8) and (5.75, 2.8) .. (a6);
	\draw[>=stealth', thick] (a1) .. controls (1.25, 2.1) and (4.75, 2.1) .. (a5);
	\draw[>=stealth', thick] (a5) .. controls (5.25, 0.7) and (6.75, 0.7) .. (a7);
	\draw[>=stealth', thick] (a3) .. controls (3.25, 2.1) and (6.75, 2.1) ..  (a7);
	\draw[>=stealth', thick] (a2) .. controls (2.25, 2.8) and (6.75, 2.8) ..  (a7);
\end{scope}

\begin{scope}[xshift=0, yshift=0, scale=1]
	\vertex[fill,label=below:\small{$1$}](a1) at (1,0) {};
	\vertex[fill,label=below:\small{$2$}](a) at (3.5,0) {};
	\vertex[fill,label=below:\small{$3$}](a6) at (6,0) {};

	\draw[>=stealth', thick] (a1)--(a);
	\draw[>=stealth', thick] (a)--(a6);
	\draw[>=stealth', thick] (a1) .. controls (1.25, .5) and (3.25, .5) .. (a);
	\draw[>=stealth', thick] (a1) .. controls (1.25, 1) and (3.25, 1) .. (a);
	\draw[>=stealth', thick] (a1) .. controls (1.25, 1.5) and (3.25, 1.5) .. (a);
	\draw[>=stealth', thick] (a1) .. controls (1.25, 2) and (3.25, 2) .. (a);
	\draw[>=stealth', thick] (a) .. controls (3.75, 0.5) and (5.75, 0.5) .. (a6);
	\draw[>=stealth', thick] (a) .. controls (3.75, 1) and (5.75, 1) ..  (a6);
	\draw[>=stealth', thick] (a) .. controls (3.75, 1.5) and (5.75, 1.5) ..  (a6);
\end{scope}

\begin{scope}[xshift=20, yshift=40, scale=1]
\node[] (a) at (0,0) {$G_{4,3}$};
\end{scope}

\begin{scope}[xshift=250, yshift=40, scale=1]
\node[] (a) at (0,0) {$G$};
\end{scope}
\end{tikzpicture}
\end{center}
\caption{The flow polytopes of the above graphs are integrally equivalent to $\Delta_4\times \Delta_3$. For visual clarity, the arrow heads are omitted for edges incident to the source $s$ and sink $t$. Contracting the bolded inner edges in $G$ yields $G_{4,3}$} 
\label{fig:prodOfSimplices}
\end{figure}

A useful tool for subdividing certain flow polytopes is M\'esz\'aros' subdivision algebra.  
It was first used to study subdivisions of acyclic root polytopes \cite{Mes11}, and has since also been extensively used to study subdivisions of flow polytopes \cite{Mes14,Mes16,MS17}.
Ceballos, Padrol, and Sarmiento studied the geometry of Tamari lattices in \cite{CPS19} via triangulations of a subpolytope $\calU_{I, \overline{J}}$ of a product of two simplices 
$\Delta_{\abs{I}} \times$ $\Delta_{\abs{\overline{J}}}$. 
They noted that $\calU_{I,\overline{J}}$ can be projected into a family of polytopes which include acyclic root polytopes, and they were interested to know if the subdivision algebra can be used to obtain subdivisions of $\calU_{I,\overline{J}}$. 
This was shown to be possible by the present author and Yip in a recent preprint titled ``On the subdivision algebra for the polytope $\calU_{I,\overline{J}}$''.  
Since $\calU_{I,\overline{J}}$ is a subpolytope of $\Delta_{\abs{I}}\times \Delta_{\abs{\overline{J}}}$, it is natural to wonder if the subdivision algebra can be extended further to products of two simplices.
In Section~\ref{sec:4} we give such an extension, which can be used to obtain refinements of the simplex-subdivision induced by $\nu$. 
In particular, in Theorem~\ref{thm:polynomial} we encode the simplex-subdivision of $\calF_{\widehat{G}_B(\nu)}$ with a polynomial $P_\nu$, and show that its reductions encode refinements of the simplex-subdivision induced by $\nu$. 

The full-dimensional cells in the simplex-subdivision of $\calF_{\widehat{G}_B(\nu)}$ arise as flow polytopes over graphs generated by cyclic shifting the path $\overline{\nu}$. 
Obtaining subdivisions for a product of simplices by cyclically shifting a path was first considered by Ceballos, Padrol, and Sarmiento in \cite{CPS15}, where they introduce the {\em Dyck path triangulation} of $\Delta_a\times \Delta_a$ and the {\em generalized Dyck path triangulation} of $\Delta_{ka}\times \Delta_a$.
In \cite{CPS19} they further generalize these triangulations to the {\em $\nu$-cyclohedral triangulation} of $\Delta_a\times \Delta_b$, which is dual to the $\nu$-cyclohedron. 
We conclude the article by showing how to obtain these triangulations using the subdivision algebra. 
Specifically, in Theorem~\ref{thm:lengthReduction} 
we give a natural reduction order for the polynomial $P_\nu$ which yields the $\nu$-cyclohedral triangulation of $\Delta_a\times \Delta_b$. 
In the special case when $\nu=E^aN^b$, this reduction order was shown in \cite[Theorem 1.3]{BGMY} to be a staircase triangulations of $\Delta_a\times \Delta_b$.

\section{Flow polytopes and negative flows}
\label{sec:2}

The typical definitions of a flow polytope encountered in the literature \cite{BV08, MS17,MMS19,DKK12,Mes15} do not allow negative flows or bidirectional edges in their underlying graphs. In this section, we remedy the situation by giving a slightly modified definition.
We will see that flow polytopes with the modified definition are integrally equivalent to the traditional flow polytopes, which might suggest that the new definition given is unnecessary.  
However, the extended definition proves to be pivotal for our purposes in the remainder of the article.

Let $G$ be a connected directed acyclic graph with vertex set $V(G)=\{1,2,\ldots, n\}$ and edge multiset $E(G)$ with $m$ directed edges. An edge $(i,j)$ with $i<j$ can be of three possible types: a \textbf{forward edge} $(i,j,+)$ directed toward $j$, a \textbf{backward edge} $(i,j,-)$ directed toward $i$, or a \textbf{bidirectional edge} $(i,j,\pm)$. 
For a vertex $j$, we let $S(j)$ denote the set of \textbf{smaller} edges of the form $(i,j,+)$, $(i,j,-)$, or $(i,j,\pm)$ with $i<j$, and we let $L(j)$ denote the set of \textbf{larger} edges of the form $(j,k,+)$, $(j,k,-)$, or $(j,k,\pm)$ with $j<k$. 
Define the \textbf{unit flow vector} $\bu = (u_1,...,u_n)$ to be the vector $\be_1 - \be_n = (1,0,...,0,-1) \in \bbZ^n$. 
At each vertex $i\in V(G)$ we assign the netflow $u_i$.
A \textbf{unit flow} (or $\bu$-flow) on $G$ is a tuple $(x_e)_{e\in E} \in \bbR^m$ satisfying

\begin{equation}
\label{eq:flowPreserved}
\sum_{e \in L(j)} x_e - \sum_{e \in S(j)}  x_e = u_j
\end{equation}
with $0 \leq x_e \leq 1$ if $e$ is a forward edge, $-1\leq x_e \leq 0$ if $e$ is a backward edge, and $-1\leq x_e \leq 1$ if $e$ is bidirectional. 
We think of $\abs{x_e}$ as the amount of flow in the edge $e$, with the sign of $x_e$ indicating the direction of flow.
If a bidirectional edge has positive flow, we treat it as a forward edge, and if it has negative flow, we treat it as a backward edge. 

The set of all $\bu$-flows on $G$ is denoted by $\calF_{G}$. 
We prove in Proposition~\ref{prop:GBpolytope} that $\calF_{G}$ is a convex polytope, which justfies calling $\calF_{G}$ the \textbf{flow polytope of $G$}.
When all edges in $G$ are forward edges, $G$ is necessarily acyclic, and in this case our definition of a flow polytope agrees with other definitions in the literature (see for example \cite{DKK12,MMS19}).

The only vertices with non-zero netflow in a $\bu$-flow are the vertices $1$ and $n$, and we may therefore assume that they are the only source and sink in $G$ respectively.
A vertex that is not a source or sink is an \textbf{inner vertex} and an edge not incident to the source or sink is an \textbf{inner edge}. 
A \textbf{route} $R$ in $G$ is a maximal directed path from the source to the sink.
Note that an edge in $G$ which is not in a route cannot have non-zero flow, so we can assume that all edges in $G$ are in some route. 
Furthermore, since $G$ is acyclic, the source and sink vertex are each incident to at most one edge in a route. 
We can therefore also assume that all edges incident to the source and sink are forward edges. 
Let $\calR$ denote the set of all routes in $G$.
To each route $R \in \calR$ we associate a signed characteristic vector $\bx_R = (x_e)_{e\in E(G)}$ as follows. 
If $e \in R$ is a forward edge $(i,j,+)$ or a bidirectional edge $(i,j,\pm)$ traversed from $i$ to $j$ in $R$, we set $x_e = 1$. 
If $e \in R$ is a backward edge $(i,j,-)$ or a bidirectional edge $(i,j,\pm)$ traversed from $j$ to $i$ in $R$, then we set $x_e = -1$. 
For any $e \notin R$, we set $x_e=0$. 
The following proposition now justifies calling $\calF_{G}$ a flow polytope, and gives an alternate description as the convex hull of routes. 

\begin{proposition}
\label{prop:GBpolytope}
The set of all $\bu$-flows $\calF_{G}$ is a polytope if and only if $G$ is acyclic. 
Furthermore, when $G$ is acyclic, the vertex description of $\calF_{G}$ is given by $$\calF_{G} = \mathrm{conv}\{ \bx_R \mid R \in \calR \}.$$
\end{proposition}

\begin{proof}
First, if $G$ is not acyclic, then it has a cycle $C$. 
Given any unit flow, one can arbitrarily change the flow in $C$ while preserving the conservation of flow condition~(\ref{eq:flowPreserved}).
Thus the set of $\bu$-flows is not bounded and hence not a polytope. 

If $G$ is acyclic, we show that $\calF_{G} = \mathrm{conv}\{\bx_R\mid R\in \calR\}$. 
We label the routes in $\calR$ such that $\calR = \{R_1,\ldots, R_k\}$ for some positive integer $k$, and we let $\bx_i$ denote the signed characteristic vector for $R_i$. 
Letting $\bp \in \mathrm{conv}\{\bx_R\mid R\in \calR\}$, we can now write $\bp = \sum_{i=1}^k \lambda_i \bx_i$ where $\lambda_i \in \bbR_{\geq 0}$ and $\sum_{i=1}^k \lambda_i = 1$. 
The netflow at vertex $1$ in $G$ is given by the sum $\sum \lambda_i $ where the sum is taken over only the indices of routes using an edge in $L(1)$. 
Since every route in $\calR$ must begin with an edge in $L(1)$, the netflow at vertex $1$ is $\sum_{i=1}^k \lambda_i= 1$. 
Similarly we obtain that the netflow at vertex $n$ is $-1$. 
Furthermore, since each $\bx_i$ adds no netflow at any inner vertex we see that any linear combination $\sum_{i=1}^k \lambda_i \bx_i$ also has netflow $0$ at each inner vertex. 
Thus $\bp \in \calF_{G}$.

Let $\bq_1$ be a $\bu$-flow in $\calF_{G}$, and let $q_{1,e}$ denote the coordinate of $\bq_1$ corresponding to the edge $e$. 
We construct $\bq_1$ as a linear combination of routes in an iterative fashion. 
Consider an edge $e_1\in R_1$ such that $\abs{q_{1,e_1}}$ is minimal, and let $\lambda_1 = \abs{q_{1,e_1}}$. 
Then $\bq_1 - \lambda_1\bx_1$ is a $(1-\lambda_1)\bu$-flow. 
Let $\bq_2 = \bq_1 - \lambda_1\bx_1$ and consider an edge $e_2\in R_2$ such that $\abs{q_{2,e_2}}$ is minimal and let $\lambda_2 = \abs{q_{2,e_2}}$. 
We obtain that $\bq_1- \lambda_1\bx_1 - \lambda_2\bx_2$ is a $(1-\lambda_1-\lambda_2)\bu$-flow. 
We continue this process, constructing a $(1-\sum_{i=1}^j\lambda_i)\bu$-flow $\bq_j$ for each $R_j \in \calR$ by removing the maximal amount of flow along the route $R_j$.
At each step the netflow of inner vertices is unaffected, while the netflow at the source decreases by $\lambda_i$ and the netflow at the sink increases by $\lambda_i$.
We continue this process until we obtain that $\bq_1 - \sum_{i=1}^k \lambda_i \bx_i$ is a $(1-\sum_{i=1}^k\lambda_i)\bu$-flow. 
If there remains an edge $e$ in $G$ with non-zero flow in the $(1-\sum_{i=1}^k\lambda_i)\bu$-flow, then by conservation of flow $e$ is either in a cycle or a route. 
Since $G$ is acyclic, $e$ must be in a route $R_j$ with non-zero flow in the $(1-\sum_{i=1}^k\lambda_i)\bu$-flow on $G$. 
This contradicts the fact that the largest possible flow in $R_j$ was subtracted at step $j$. 
Thus all edges in an $(1-\sum_{i=1}^k\lambda_i)\bu$-flow
on $G$ are zero.
It then follows that $\bq_1 = \sum_{i=1}^k \lambda_i\bx_i$ and $\sum_{i=1}^k  \lambda_i = 1$, and therefore $\calF_{G} = \mathrm{conv}\{\bx_R \mid R\in \calR\}$. 

It remains to show that $\bx_R$ is a vertex of $\calF_{G}$ for each $R\in \calR$.
It suffices to show that no three vectors in $\{\bx_R\mid R\in \calR\}$ are colinear. 
Suppose toward a contradiction that there are distinct $u$, $v$, and $w$ in $\{1, \ldots, k\}$ such that $\bx_u$, $\bx_v$ and $\bx_w$ are colinear. 
Then $\bx_u-\bx_w = c(\bx_v - \bx_w)$ for some $c\in \bbR$, and so $\bx_u = c\bx_v + (1- c)\bx_w$.
Since $\bx_u$, $\bx_v$, and $\bx_w$ are distinct, there is a first edge $e$ along $R_v$ that is not in $R_w$. 
Hence $x_{v,e} = \pm 1$ and $x_{w,e} = 0$.
Thus $x_{u,e} = cx_{v,e} = \pm \, c\, $, where $x_{u,e} \in \{-1,0,1\}$. 
If $x_{u,e} = 0$, then $c=0$, which implies that $\bx_u = \bx_w$ and hence contradicts the fact that the vectors are distinct. 
Similarly, if $x_{u,e} = 1$, then $c =1$ and we reach the contradiction that $\bx_u = \bx_v$. 
If $c = -1$, then $\bx_u + \bx_v = 2\bx_w$. 
Since $R_u$ and $R_v$ are distinct, there is a first edge $e'$ in $R_u$ that is not in $R_v$. 
Therefore $x_{u,e'}= 1$ and $x_{v,e'} = 0$. 
It follows that $1 = 2x_{w,e'}$, which is a contradiction since $x_{w,e'}$ must be an integer. 
\end{proof}

Having now shown that our notion of a flow polytope is well-defined, we can make explicit
its connection to the usual definition where all edges in the underlying graph are assumed to be forward edges.
Two lattice polytopes $\calP\subseteq \bbR^k$ and $\calQ\subseteq \bbR^\ell$ are said to be \textbf{integrally equivalent} if there exists an affine transformation $\varphi:\bbR^k\rightarrow \bbR^\ell$ whose restriction to $\calP$ is a bijection $\varphi:\calP\to \calQ$ that preserves the lattice, meaning $\varphi$ is a bijection between $\bbZ^k \cap \mathrm{aff}(P)$ and $\bbZ^\ell \cap \mathrm{aff}(Q)$.
If $\calP$ and $\calQ$ are integrally equivalent, they have the same face poset, normalized volume and Ehrhart polynomial, and we write $\calP\equiv\calQ$.

\begin{lemma}
Let $G$ be an acyclic digraph. Then $\calF_G$ is integrally equivalent to a flow polytope $\calF_{G^*}$, where all edges of $G^*$ are forward edges.
\label{lem:intEq}
\end{lemma}

\begin{proof}
Since $G$ is acyclic, the vertex labels can be permuted such that $G$ has no backward edges. We let $G^*$ be the graph obtained from $G$ by contracting all bidirectional edges and then relabeling so that there are no backward edges. 
As a result, $G^*$ only has forward edges, and it remains to show that the operations performed are integral equivalences.
Changing the orientation of a backward edge $e$ in $G$ amounts to reflecting the flow polytope about the hyperplane $x_e=0$, which is an integral equivalence.
Next, consider the operation of contracting a bidirectional edge $(i,j,\pm)$ with $i<j$, where $(i,j,\pm)$ is the $k$th edge in a fixed ordering of the edges of $G$.
Define the map $\varphi:\bbR^m \to \bbR^{m-1}$ by 
$$\varphi(x_1,\ldots, x_k, \ldots, x_m) = (x_1,\ldots, \widehat{x}_k, \ldots, x_m),$$
and the map $\psi:\bbR^{m-1}\to \bbR^m$ by 
$$\psi(x_1,\ldots, \widehat{x}_k,\ldots, x_m)= (x_1,\ldots, x_{k-1},\sum_{e\in L(i)\setminus (i,j,\pm)}x_e - \sum_{e\in S(i)} x_e, x_{k+1}, \ldots, x_m).$$
Now restricting $\varphi$ to $\calF_{G}$ gives a bijection between $\calF_{G}$ and $\calF_{G/(i,j,\pm)}$, with its inverse given by restricing $\psi$ to $\calF_{G/(i,j,\pm)}$. 
It is straightforward to verify that $\varphi$ restricts to a bijection between $\mathbb{Z}^m \cap \mathrm{aff}(\calF_{G})$ and $\mathbb{Z}^{m-1} \cap \mathrm{aff}(\calF_{G/(i,j,\pm)})$.
\end{proof}

An immediate consequence of Lemma~\ref{lem:intEq} is that the dimension formula $\dim(\calF_G) = \abs{E(G)} - \abs{V(G)} + 1$ for flow polytopes (see for example \cite[Proposition 1]{DKK12}) also holds for our definition of flow polytope. 

An edge $(i,j)$ is said to be \textbf{idle} if it is the only outgoing edge from $i$, or it is the only incoming edge to $j$.
The following is another useful operation that preserves flow polytopes up to integral equivalence.

\begin{lemma}(\cite[Lemma 2.2]{MS17})
\label{lem:idleEdgeContraction}
If $G'$ is obtained from $G$ by contracting idle edges, then $\calF_{G'} \equiv \calF_{G}$. \qed
\end{lemma}

\section{The flow polytope $\calF_{\widehat{G}_B(\nu)}$ and its simplex subdivision}
\label{sec:3}

In this section we use a lattice path $\nu$ to construct the main flow polytope of consideration in this article. We show that the path $\nu$ also induces a polyhedral subdivision of the flow polytope that is dual to a simplex, with each interior faces in the subdivision being integrally equivalent to a flow polytope.

\subsection{The flow polytope $\calF_{\widehat{G}_B(\nu)}$}
Let $a$ and $b$ be nonnegative integers, and let $\nu$ be a lattice path from $(0,0)$ to $(a,b)$, with steps $E=(1,0)$ and $N = (0,1)$. 
Let $\overline{\nu} := E\nu N$, and consider it as a word on $a+b+2$ letters in the alphabet $\{E,N\}$. We write $\overline{\nu} =\prod_{i=1}^{w} E^{a_i}N^{b_i}$, where $a_i>0$ and $b_i>0$ for $1\leq i\leq w$, where $w$ is the number of valleys (consecutive $EN$ pairs) in $\overline{\nu}$.
To distinguish between the $E$ steps and $N$ steps in $\overline{\nu}$, we index them as follows.
Reading the letters in $\overline{\nu}$ from left to right, each letter is indexed with the next positive integer, with the exception that the first $N$ in a run of $N$ steps receives the index of the preceding $E$ step. 
We refer to such an indexing as the \textbf{canonical indexing of} $\overline{\nu}$. In addition, we refer to any indexing of $\overline{\nu}$ obtained by cyclically permuting the labels in the canonical indexing as a \textbf{valid indexing}.
For example, $E_1N_1E_2E_3N_3N_4E_5N_5$ is a canonical indexing, while $E_5N_5E_1E_2N_2N_3E_4N_4$ is a valid indexing of the same path. 
Let $I$ be the set of indices of $E$ steps and let $J$ be the set of indices of $N$ steps.
We have that $\abs{I} = a+1 = \sum_{i=1}^w a_i$ and $\abs{J} = b+1 = \sum_{i=1}^w b_i$.
Furthermore, the valleys of $\overline{\nu}$ are indexed by elements in $I \cap J$.

\begin{definition}
For a path $\overline{\nu}$ with its letters indexed by $\{1,...,n\}$ according to a valid indexing, we define the \textbf{$\nu$-graph}, denoted $G(\nu)$, to be the digraph on vertex set $[n]$ and edge set defined as follows.
A pair $(i,j)$ is a directed edge in $G(\nu)$ toward $j$ if and only if  $E_i\cdots N_j$ is a subword of $\overline{\nu}$ and its only valleys are of the form $E_kN_k$ with $k=i$ or $k=j$. 
The \textbf{bidirectional $\nu$-graph} $G_B(\nu)$ is the graph $G(\nu)$, but with edges $(i,j)$ being bidirectional whenever $i$ and $j$ are both in $I\cap J$. 
\end{definition}

By construction, the bidirectional edges in $G_B(\nu)$ form a path of length $w-1$. 
Note that when $\overline{\nu}$ has only one valley, there are no edges whose end points are in $I\cap J$, and so $G_B(\nu) = G(\nu)$. 
The number of edges in $G(\nu)$ is $\abs{I\cup J}-1$, as it contains an edge for each element of $I\setminus J$, an edge for each element of $J\setminus I$, and the edges on the path through vertices in $I\cap J$.

\begin{center}
\begin{figure}
\captionsetup{justification=centering}
\begin{tikzpicture}
\begin{scope}[xshift=0, yshift=-10, scale=0.7]
	\vertex[fill,label=below:\small{$1$}](a1) at (1,0) {};
	\vertex[fill,label=below:\small{$2$}](a2) at (2,0) {};
	\vertex[fill,label=below:\small{$3$}](a3) at (3,0) {};
	\vertex[fill,label=below:\small{$4$}](a4) at (4,0) {};
	\vertex[fill,label=below:\small{$5$}](a5) at (5,0) {};
	\vertex[fill,label=below:\small{$6$}](a6) at (6,0) {};

	\draw[>=stealth', ->, very thick] (a2)--(a3);
	\draw[>=stealth', ->, very thick] (a3)--(a4);
	\draw[>=stealth', ->, very thick] (a5)--(a6);
	\draw[>=stealth', <->, very thick] (a1) .. controls (1.45, 0.8) and (2.55, 0.8) .. (a3);
	\draw[>=stealth', <->, very thick] (a3) .. controls (3.45, 1) and (5.55, 1) .. (a6);	
\end{scope}
\begin{scope}[xshift=210, yshift=-10, scale=0.8]
	\vertex[fill,label=below:\small{$s$}](as) at (0,0) {};
	\vertex[fill,label=below:\small{$1$}](a1) at (1,0) {};
	\vertex[fill,label=below:\small{$2$}](a2) at (2,0) {};
	\vertex[fill,label=below:\small{$3$}](a3) at (3,0) {};
	\vertex[fill,label=below:\small{$4$}](a4) at (4,0) {};
	\vertex[fill,label=below:\small{$5$}](a5) at (5,0) {};
	\vertex[fill,label=below:\small{$6$}](a6) at (6,0) {};
	\vertex[fill,label=below:\small{$t$}](at) at (7,0) {};

	\draw[>=stealth', ->, very thick] (a2)--(a3);
	\draw[>=stealth', ->, very thick] (a3)--(a4);
	\draw[>=stealth', ->, very thick] (a5)--(a6);
	\draw[>=stealth', <->, very thick] (a1) .. controls (1.45, 0.8) and (2.55, 0.8) .. (a3);
	\draw[>=stealth', <->, very thick] (a3) .. controls (3.45, 1) and (5.55, 1) .. (a6);
	\draw[>=stealth',->,color=gray] (as)--(a1);	
	\draw[>=stealth',->,color=gray] (as) .. controls (0.25, 0.75) and (1.75, 0.75) .. (a2);
	\draw[>=stealth',->,color=gray] (as) .. controls (0.25, 1.5) and (2.75, 1.5) .. (a3);
	\draw[>=stealth',->,color=gray] (as) .. controls (0.25, 2.25) and (4.75, 2.25) .. (a5);	\draw[>=stealth',->,color=gray] (as) .. controls (0.25, 3.0) and (5.75, 3.0) .. (a6);
	
	\draw[>=stealth',->,color=gray] (a6)--(at);	
	\draw[>=stealth',->,color=gray] (a1) .. controls (1.25, -3) and (6.75, -3) .. (at);
	\draw[>=stealth',->,color=gray] (a3) .. controls (3.25, -2.25) and (6.75, -2.25) .. (at);
	\draw[>=stealth',->,color=gray] (a4) .. controls (4.25, -1.5) and (6.75, -1.5) .. (at);	
\end{scope}

\begin{scope}[xshift=70, yshift=-40, scale=0.8]
	\node[] (hi) at (0,0) {$\overline{\nu} = E_1N_1E_2E_3N_3N_4E_5E_6N_6$};
\end{scope}

\begin{scope}[xshift=7, yshift=20, scale=0.8]
	\node[] (hi) at (0,0) {$G_B(\nu)$};
\end{scope}

\begin{scope}[xshift=195, yshift=20, scale=0.8]
	\node[] (hi) at (0,0) {$\widehat{G}_B(\nu)$};
\end{scope}
\end{tikzpicture}
\caption{The graphs $G_B(\nu)$ and $\widehat{G}_B(\nu)$ for the path $\nu = NEENNEE$} 
\label{fig:G(nu)}
\end{figure}
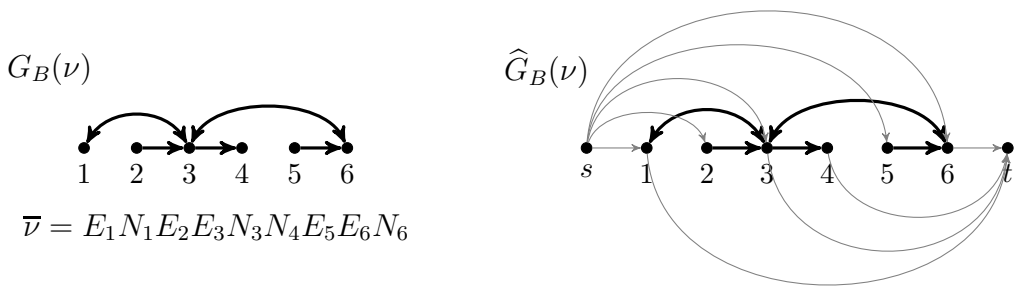
\end{center}

\begin{definition}
For a digraph $G$ on the vertex set $[n]$, define its \textbf{full augmentation} to be the graph $\widetilde{G}$ obtained by adding to $G$ a source $s$ and sink $t$ along with the set of forward edges $\{(s,i,+), (i,t,+)\mid i\in [n]\}$.
If $G$ is a digraph on the vertex set $I\cup J$, where $I$ and $J$ are respectively the index sets of the $E$ and $N$ steps of a canonically indexed $\overline{\nu}$, then we define the \textbf{$\nu$-augmentation} of $G$ to be the graph $\widehat{G}^\nu$ obtained by adding to $G$ a source $s$ and sink $t$, along with the edges $\{(s,i,+) \mid i \in I\}$ and $\{(j,t,+) \mid j\in J\}$. 
When $\nu$ is clear from context, we simply write $\widehat{G}$ for the $\nu$-augmentation of $G$. 
\end{definition}

An example of a graph $G_B(\nu)$ and its $\nu$-augmentation $\widehat{G}_B(\nu)$ are given in Figure~\ref{fig:G(nu)}.

\begin{lemma}
\label{lem:intEq2}
Let $\nu$ be a lattice path from $(0,0)$ to $(a,b)$. The flow polytope $\calF_{\widehat{G}_B(\nu)}$ is integrally equivalent to the product of simplices $\Delta_a \times \Delta_b$. 
\end{lemma}

\begin{proof}
By construction, an inner edge of $\widehat{G}_B(\nu)$ is either idle or bidirectional. 
Thus contracting the idle inner edges and bidirectional edges yields the graph $G_{a,b}$ on vertex set $\{s,1,t\}$ with $\abs{I} = a+1$ edges of the form $(s,1,+)$ and $\abs{J} = b+1$ edges of the form $(1,t,+)$. 
By Lemma~\ref{lem:intEq} and Lemma~\ref{lem:idleEdgeContraction} these contractions preserve the flow polytopes up to integral equivalence.
Hence $\calF_{\widehat{G}_B(\nu)} \equiv \calF_{G_{a,b}} = \Delta_a\times \Delta_b$. 
\end{proof}

\subsection{The simplex-subdvision of $\calF_{\widehat{G}_B(\nu)}$}
Define a \textbf{cyclic peak} of a lattice path to be a consecutive $NE$ pair in a cyclic reading of the steps in the path. 
In particular, the only cyclic peak which is not a peak occurs when the path begins with $E$ and ends with $N$.
Given a path $\nu$, we define $\overline{\nu}(i)$ to be the path obtained by reading the steps of $\overline{\nu}$ beginning at the $E$ step of the $i$th cyclic peak.
We begin counting from the first cyclic peak which is also a peak of $\overline{\nu}$. 
Since $\overline{\nu}$ begins with $E$ and ends with $N$, the number of cyclic peaks is the same as the number of valleys $w$ in $\overline{\nu}$, with the first and last steps of $\overline{\nu}$ forming the $w$th cyclic peak. 
For example, the path $\overline{\nu} = E_1N_1E_2N_2E_3E_4N_4E_5N_5$ has four cyclic peaks, and 
\begin{align*}
\overline{\nu}(1) &= E_2N_2E_3E_4N_4E_5N_5E_1N_1, \\
\overline{\nu}(2) &= E_3E_4N_4E_5N_5E_1N_1E_2N_2, \\
\overline{\nu}(3) &= E_5N_5E_1N_1E_2N_2E_3E_4N_4,\\
\overline{\nu} = \overline{\nu}(4) &= E_1N_1E_2N_2E_3E_4N_4E_5N_5.
\end{align*}
The labeling of each cyclically shifted path $\overline{\nu}(i)$ is a valid labeling, and hence induces a graph $G(\nu(i))$, where $\nu(i)$ denotes $\overline{\nu}(i)$ without the initial $E$ step and terminal $N$ step. 
Note that different valid indexings of $\overline{\nu}(i)$ give rise to different cyclic permutations of the vertices of $G(\nu(i))$. 
In particular, we have the following.

\begin{lemma}
Let $p_i$ be the index of the $N$ step in the $i$th cyclic peak of $\overline{\nu}$. 
The graph obtained by the relabeling of the vertices of $G(\nu(i))$ by $\sigma_i:k\mapsto k-p_i \pmod n$ has only forward edges.
\label{lem:sigma}
\end{lemma}

\begin{proof}
Applying $\sigma_i:k\mapsto k-p_i \pmod n$ to the indices of $\overline{\nu}(i)$ give its canonical indexing. The edges of $G(\nu)$ are forward edges for any canonically indexed $\overline{\nu}$. 
\end{proof}

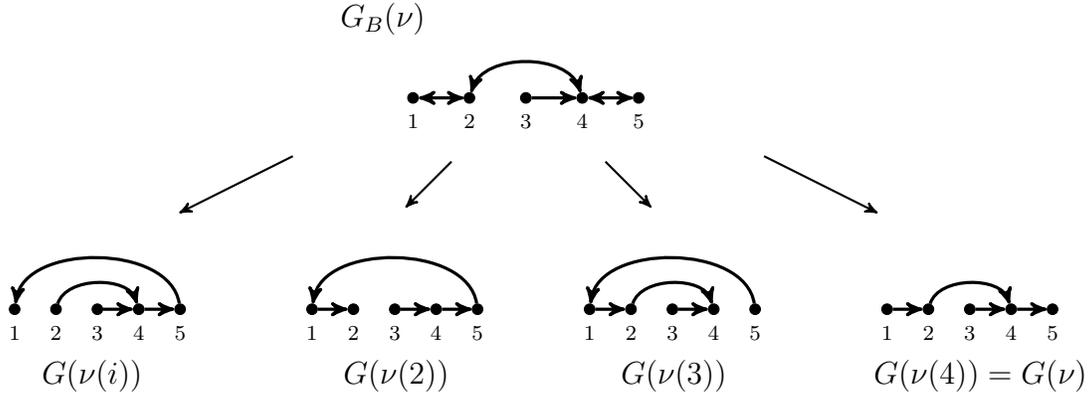
\begin{figure}
\begin{center}
\begin{tikzpicture}
\begin{scope}[xshift=-20, yshift=0, scale=0.75]
	\vertex[fill,label=below:\tiny{$1$}](a1) at (1,0) {};
	\vertex[fill,label=below:\tiny{$2$}](a2) at (2,0) {};
	\vertex[fill,label=below:\tiny{$3$}](a3) at (3,0) {};
	\vertex[fill,label=below:\tiny{$4$}](a4) at (4,0) {};
	\vertex[fill,label=below:\tiny{$5$}](a5) at (5,0) {};

	\draw[>=stealth', <->, very thick] (a1)--(a2);
	\draw[>=stealth', <->, very thick] (a2) .. controls (2.25, 0.8) and (3.75, 0.8) .. (a4);
	\draw[>=stealth', ->, very thick] (a3)--(a4);
	\draw[>=stealth', <->, very thick] (a4)--(a5);

\end{scope}

\begin{scope}[xshift=165, yshift=-80, scale=0.55]
	\vertex[fill,label=below:\tiny{$1$}](a1) at (1,0) {};
	\vertex[fill,label=below:\tiny{$2$}](a2) at (2,0) {};
	\vertex[fill,label=below:\tiny{$3$}](a3) at (3,0) {};
	\vertex[fill,label=below:\tiny{$4$}](a4) at (4,0) {};
	\vertex[fill,label=below:\tiny{$5$}](a5) at (5,0) {};
	\draw[>=stealth', ->, very thick] (a1)--(a2);
	\draw[>=stealth', ->, very thick] (a2) .. controls (2.25, 0.8) and (3.75, 0.8) .. (a4);
	\draw[>=stealth', ->, very thick] (a3)--(a4);
	\draw[>=stealth', ->, very thick] (a4)--(a5);
\end{scope}

\begin{scope}[xshift=-165, yshift=-80, scale=0.55]
	\vertex[fill,label=below:\tiny{$1$}](a1) at (1,0) {};
	\vertex[fill,label=below:\tiny{$2$}](a2) at (2,0) {};
	\vertex[fill,label=below:\tiny{$3$}](a3) at (3,0) {};
	\vertex[fill,label=below:\tiny{$4$}](a4) at (4,0) {};
	\vertex[fill,label=below:\tiny{$5$}](a5) at (5,0) {};
	\draw[>=stealth', ->, very thick] (a2) .. controls (2.25, 0.8) and (3.75, 0.8) .. (a4);
	\draw[>=stealth', ->, very thick] (a3)--(a4);
	\draw[>=stealth', ->, very thick] (a4)--(a5);
	\draw[>=stealth', <-, very thick] (a1) .. controls (1.25, 1.6) and (4.75, 1.6) .. (a5);
\end{scope}

\begin{scope}[xshift=-52.5, yshift=-80, scale=0.55]
	\vertex[fill,label=below:\tiny{$1$}](a1) at (1,0) {};
	\vertex[fill,label=below:\tiny{$2$}](a2) at (2,0) {};
	\vertex[fill,label=below:\tiny{$3$}](a3) at (3,0) {};
	\vertex[fill,label=below:\tiny{$4$}](a4) at (4,0) {};
	\vertex[fill,label=below:\tiny{$5$}](a5) at (5,0) {};

	\draw[>=stealth', ->, very thick] (a1)--(a2);

	\draw[>=stealth', ->, very thick] (a3)--(a4);
	\draw[>=stealth', ->, very thick] (a4)--(a5);
	\draw[>=stealth', <-, very thick] (a1) .. controls (1.25, 1.6) and (4.75, 1.6) .. (a5);
\end{scope}

\begin{scope}[xshift=52.5, yshift=-80, scale=0.55]
	\vertex[fill,label=below:\tiny{$1$}](a1) at (1,0) {};
	\vertex[fill,label=below:\tiny{$2$}](a2) at (2,0) {};
	\vertex[fill,label=below:\tiny{$3$}](a3) at (3,0) {};
	\vertex[fill,label=below:\tiny{$4$}](a4) at (4,0) {};
	\vertex[fill,label=below:\tiny{$5$}](a5) at (5,0) {};

	\draw[>=stealth', ->, very thick] (a1)--(a2);
	\draw[>=stealth', ->, very thick] (a2) .. controls (2.25, 0.8) and (3.75, 0.8) .. (a4);
	\draw[>=stealth', ->, very thick] (a3)--(a4);
	\draw[>=stealth', <-, very thick] (a1) .. controls (1.25, 1.6) and (4.75, 1.6) .. (a5);
\end{scope}

\begin{scope}[xshift=-10, yshift=30, scale=0.8]
\node[] (a) at (0,0) {$G_B(\nu)$};
\end{scope}

\begin{scope}[xshift=216, yshift=-105, scale=0.8]
\node[] (a) at (0,0) {$G(\nu(4)) = G(\nu)$};
\end{scope}

\begin{scope}[xshift=-120, yshift=-105, scale=0.8]
\node[] (a) at (0,0) {$G(\nu(i))$};
\end{scope}

\begin{scope}[xshift=-05, yshift=-105, scale=0.8]
\node[] (a) at (0,0) {$G(\nu(2))$};
\end{scope}

\begin{scope}[xshift=100, yshift=-105, scale=0.8]
\node[] (a) at (0,0) {$G(\nu(3))$};
\end{scope}

\begin{scope}[xshift=-40, yshift=-20, scale=0.9]
\node[] (a1) at (0,0) {};
\node[] (a2) at (-2,-1) {};
\draw[>=stealth', ->, thick] (a1)--(a2);
\end{scope}
\begin{scope}[xshift=130, yshift=-20, scale=0.9]
\node[] (a1) at (0,0) {};
\node[] (a2) at (2,-1) {};
\draw[>=stealth', ->, thick] (a1)--(a2);
\end{scope}
\begin{scope}[xshift=20, yshift=-20, scale=0.9]
\node[] (a1) at (0,0) {};
\node[] (a2) at (-1,-1) {};
\draw[>=stealth', ->, thick] (a1)--(a2);
\end{scope}
\begin{scope}[xshift=70, yshift=-20, scale=0.9]
\node[] (a1) at (0,0) {};
\node[] (a2) at (1,-1) {};
\draw[>=stealth', ->, thick] (a1)--(a2);
\end{scope}

\end{tikzpicture}
\end{center}
\caption{Generating the possible graphs $G(\nu(i))$ from $G_B(\nu)$ with $\nu = NENEENE$}
\label{fig:carBtoQ}
\end{figure}

The graphs $G(\nu(i))$ can be obtained from $G_B(\nu)$ without reference to $\nu$.
Let $v_1 < v_2 < \cdots < v_w$ be the vertices of the bidirectional path in $G_B(\nu)$. 
Let $C$ denote the cycle graph with edges $\{(v_1,v_2,+)$, $(v_2,v_3,+)$, $\ldots$ , $(v_{w-1},v_{w},+)$, $(v_1,v_w,-)\}$, and let $C(i)$ denote the graph $C$ without the edge directed away from $v_i$. 
The graph $G(\nu(i))$ is then obtained by replacing the bidirectional path in $G_B(\nu)$ with $C(i)$. 
Note in particular that $G(\nu(w)) = G(\nu)$ since $C(w)$ is the directed path from $v_1$ to $v_w$.
See Figure~\ref{fig:carBtoQ} for an example.

Let $v_i\in I\cap J$ be the index of the $i$th valley in the canonical indexing of $\overline{\nu}$. 
For example, the path $\overline{\nu} = \overline{\nu}(4)$ above has $I\cap J = \{1,2,4,5\}$, so $v_1=1$, $v_2=2$, $v_3 = 4$ and $v_4=5$. 
For $1\leq i < w$, let $\calR_i$ denote the set of routes in $\widehat{G}_B(\nu)$ defined as follows. 
A route $R\in \calR_i$ if and only if 

\begin{itemize}
    \item[1.] $R$ does not traverse the edge $(v_i,v_{i+1},\pm)$ as a forward edge; and  
    \item[2.] if $R$ has backward edges, then $R$ traverses the edge $(v_i,v_{i+1},\pm)$ as a backward edge. 
\end{itemize}

We then define $\calQ_{w} := \widehat{G}(\nu)$ and for each $1\leq i < w$ we define $$ \calQ_i := \mathrm{conv}\{\bx_R\mid R\in \calR_i\}.$$

For example, $\calQ_1$ for $\widehat{G}_B(\nu)$ in Figure~\ref{fig:G(nu)} is the convex hull of all routes in $\widehat{G}_B(\nu)$ except the routes $(s,6,3,t)$, $(s,6,3,4,t)$, and any route using edge $(1,3,\pm)$ as a forward edge.

\begin{theorem}
\label{thm:subdivision}
The flow polytope $\calF_{\widehat{G}_B(\nu)}$ can be written as
$$\calF_{\widehat{G}_B(\nu)} = \bigcup_{i=1}^w \calQ_i $$ 
where the polytopes $\calQ_1, \calQ_2$,..., $\calQ_w$ have pairwise disjoint interiors, and $w$ is the number of valleys in $\overline{\nu}=E\nu N$. Furthermore, for any nonempty $S \subseteq [w]$ we have that $\cap_{i\in S} \calQ_i\equiv \calF_{\cap_{i\in S}\widehat{G}(\nu(i)) }$.
\end{theorem}

\begin{proof}
We begin by checking that $\calF_{\widehat{G}_B(\nu)} = \bigcup_{i=1}^w \calQ_i$. 
Let $\bx_1 \in \calF_{\widehat{G}_B(\nu)}$. First, if all entries in $\bx_1$ are nonnegative, then $\bx_1 \in \calF_{\widehat{G}(\nu)} = \calQ_w$ so it can be written as a convex combination of routes in $\calQ_w$.
If $\bx_1$ has negative entries, then necessarily $w>1$ and the set $I\cap J=\{v_1,...,v_w\}$ contains at least two elements. 
We assume that $v_1<v_2<\cdots < v_w$.
Of the edges with negative flow, we choose edges $(v_{h_1},v_{h_1+1},\pm)$ and $(v_{h_1 + r_1-1},v_{h_1+ r_1},\pm)$ (not necessarily distinct) such that $v_{h_1}$ is minimal and $v_{h_1+r_1}$ is maximal.
The edges $(v_{h_1},v_{h_1 +1},\pm)$, $(v_{h_1+1},v_{h_1 +2},\pm)$, $\ldots$, $(v_{h_1+r_1-1},v_{h_1+r_1},\pm)$ then form the shortest path $P_1$ such that any edge with negative flow occurs in $P_1$ (although $P_1$ may contain edges with nonnegative flow). 

By conservation of flow at $v_{h_1}$, there must be an edge of the form $(s_1,v_{h_1 + r_1},+)$ with positive flow. 
By the construction of $\widehat{G}_B(\nu)$, either $s_1=s$ or there is an edge $(s,s_1,+)$ satisfying
$x_{(s,s_1,+)} = x_{(s_1,v_{h_1 + r_1},+)}$.
Similarly, there must be an edge $(v_{h_1},t_1,+)$ with positive flow that extends uniquely to the sink $t$.
Let $R_1$ be the unique route using the edges $(s_1,v_{h_1 + r_1},+)$, $(v_{h_1},t_1,+)$, and the edges in $P_1$.
The route $R_1$ contains all edges with negative flow in $\bx_1$.
Letting $\alpha_1 := \min\{x_{(s_1,v_{h_1+r_1},+)},x_{(v_{h_1},t_1,+)}\}>0$, we have that $\bx_2 := \bx_1 - \alpha_1 \bx_{R_1}$ is a $(1-\alpha_1)\bu$-flow on $\widehat{G}_B(\nu)$ in which one or both of the edges $(s_1,v_{h_1+r_1},+)$ and $(v_{h_1},t_1,+)$ have zero flow.
Furthermore, the flow in the edges of $P_1$ is strictly increased in $\bx_2$. If $\bx_2$ has negative entries, we can repeat the process. 

In general, for a $(1-\sum_{k=1}^{\ell-1}\alpha_k)\bu$-flow  $\bx_{\ell}$ with $\ell\geq 2$ on $\widehat{G}_B(\nu)$ with negative entries, we let $\bx_{\ell+1}$ be the 
$(1-\sum_{k=1}^{\ell}\alpha_k)\bu$-flow obtained from $\bx_\ell$ as follows. 
Let $(v_{h_\ell},v_{h_\ell+1},\pm)$ and $(v_{h_\ell+r_\ell - 1},v_{h_\ell+r_\ell}, \pm)$ be the edges with negative flow in $\bx_\ell$ where $v_{h_\ell}$ is minimal and $v_{h_\ell+r_\ell}$ is maximal.
The path $P_\ell$ from $v_{h_\ell}$ to $v_{h_\ell+r_\ell}$ is then the shortest path containing the edges with negative flow in $\bx_\ell$.
As in the case when $\ell=1$ above, the path $P_\ell$ can be extended to a route $R_\ell$ through edges $(s_\ell,v_{h_\ell + r_\ell},+)$ and $(v_{h_\ell},t_\ell,+)$, both of which have positive flow. 
Letting $\alpha_\ell := \min\{x_{(s_\ell,v_{h_\ell+r_\ell},+)},x_{(v_{h_\ell},t_\ell,+)}\}>0$, we define $\bx_{\ell+1} := \bx_\ell - \sum_{k=1}^{\ell}\alpha_k \bx_{R_k}$.
Now $\bx_{\ell+1}$ is a $(1-\sum_{k=1}^{\ell}\alpha_k)\bu$-flow, where one or both of the edges $(s_\ell,v_{h_\ell + r_\ell},+)$ and $(v_{h_\ell},t_\ell,+)$ have zero flow, and the flow in the edges of $P_\ell$ strictly increased.

The algorithm above can be repeated until we obtain a $(1-\sum_{k=1}^{q}\alpha_k)\bu$-flow $\bx_{q+1} := \bx_q - \sum_{k=1}^{q}\alpha_k \bx_{R_k}$ with only nonnegative entries. 
As a result, $\bx_{q+1}$ is a point in the flow polytope $\calF_{\widehat{G}(\nu)}$ dilated by a factor of $(1-\sum_{k=1}^{q}\alpha_k)$. 
Hence it can be written as a convex combination $\sum_{k=q+1}^r \alpha_k\bx_{R_k}$ with $\sum_{k=q+1}^r \alpha_k = 1-\sum_{k=1}^{q}\alpha_k$.  
Let $i$ be the index of any edge in the path $P_q$, which contains the edges with negative flow in $\bx_{q}$. 
Note that $P_q \subseteq P_{q-1} \subseteq \cdots \subseteq P_1$, and so all routes $R_k$ with $1\leq k \leq q$ traverse the edge $e_i$ as a backward edge. 
Furthermore, since $\bx_{q+1}$ has zero flow on the edge $e_i$, the routes $R_{q+1}$, $\ldots$, $R_r$ do not traverse $e_i$ as a forward edge.
Now $\bx_1 = \sum_{k=1}^r \alpha_k\bx_{R_k}$ with $\sum_{k=1}^r \alpha_k = 1$ and $\alpha_k \geq 0$ for each $k$. 
Therefore $\bx_1$ is a convex combination of routes in $\calQ_i$ and we conclude that $\calF_{\widehat{G}_B(\nu)} \subseteq \cup_{i=1}^w \calQ_i$. 
The reverse inclusion is immediate. 

To see that $\calQ_i$ and $\calQ_j$ have disjoint interiors for any $i,j \in [w]$ with $i\neq j$ we observe that $\calQ_i$ and $\calQ_j$ lie in the positive and negative half-spaces of the hyperplane determined by $x_{(v_i,v_{i+1}.\pm)} = x_{(v_j,v_{j+1},\pm)}$. 
If $i=w$ and $j\neq w$, then $\calQ_i$ and $\calQ_j$ are contained in the positive and negative half-spaces of the hyperplane determined by $x_{(v_j,v_{j+1},\pm)} = 0$. 

It remains to show that for $S\subseteq [w]$ we have $\cap_{i\in S}\calQ_i \equiv \calF_{\cap_{i\in S} \widehat{G}(\nu(i))}$. 
If $w\in S$, then $\cap_{i\in S}\calQ_i = \calF_{\cap_{i\in S} \widehat{G}(\nu(i))}$ since $\widehat{G}(\nu(w)) = \widehat{G}(\nu)$.
In the case that $w \notin S$, we construct an integral equivalence $\varphi$ as follows.  
First fix an ordering $e_1,\ldots, e_{m+1}$ on the set of edges in $\widehat{G}_B(\nu) \cup (v_1,v_w,-)$ so that $e_k = (v_k,v_{k+1}, \pm)$ for each $k\in [w-1]$ and $e_w = (v_1,v_w,-)$. 
Then $\cap_{i\in S} \calQ_i$ and $\calF_{\cap_{i\in S} \widehat{G}(\nu(i))}$ both embed into $\bbR^{m+1}$, and we define $\varphi:\bbR^{m+1} \to \bbR^{m+1}$ to be the linear transformation   
$$\bx \mapsto \bx + x_{\min(S)}\be_w - x_{\min(S)} \sum_{k = 1}^{w-1} \be_k . $$

By construction, a route $R$ in $\cap_{i\in S}\calQ_i$ must either not use any of the edges of the form $(v_i,v_{i+1},\pm)$ with $i\in S$ or it must use all of them as backward edges.
Thus for the vertex $\bx_R$, either $x_i=0$ for all $i\in S$ or $x_i = -1$ for all $i\in S$. 
In both cases, $R$ is uniquely determined by its edges $e_s$ and $e_t$ incident to the source and sink respectively. 
Now $\varphi$ maps $\bx_R$ to the unique route in $\cap_{i\in S} \widehat{G}(\nu(i))$ determined by $e_s$ and $e_t$. 
It follows that $\varphi$ restricts to a bijection between $\mathrm{aff}(\cap_{i\in S}\calQ_i)$ and $\mathrm{aff}(\cap_{i\in S} \widehat{G}(\nu(i)))$, with inverse $\varphi\inv$ given by $\varphi(\bx) = \bx - x_w\be_w + x_w \sum_{i=1}^{w-1} \be_i$.
This restriction of $\varphi$ is a bijection between $\cap_{i\in S}\calQ_i$ and $\cap_{i\in S} \widehat{G}(\nu(i))$, while mapping lattice points to lattice points.  
\end{proof}

For a polytope $\calP$ we use $\vol\calP$ to denote its normalized volume, i.e. the product of $\dim(\calP)!$ and the Euclidean volume of $\calP$.
The flow polytope $\calF_{\widehat{G}(\nu)} = \calQ_w$ was first considered by M\'esz\'aros and Morales in \cite{MM19} (up to an integral equivalence), where they showed that its normalized volume is the number of weak integer compositions above a fixed composition in the dominance order \cite[Corollary 6.19]{MM19}. 
Letting $\Cat(\nu)$ denote the number of lattice paths using $N$ and $E$ steps with the same end points as $\nu$ that do not go below $\nu$, their result can be restated as follows. 

\begin{lemma}(\cite[Theorem 1.1]{BGMY})
\label{lem:nuCarVol}
For a lattice path $\nu$, we have $\vol\calF_{\widehat{G}(\nu)} = \Cat(\nu).$
\end{lemma}

As a consequence we obtain the following identity, which is also implicit in \cite{CPS19}. 

\begin{corollary}
\label{cor:nuCatIdentity}
Let $\nu$ be a lattice path from $(0,0)$ to $(a,b)$ and let $w$ be the number of valleys in $\overline{\nu}$. Then 
$$ \binom{a+b}{a} = \sum_{i=1}^w \Cat(\nu(i)). $$
\end{corollary}

\begin{proof}
The normalized volume of $\Delta_a\times \Delta_b$ is $\binom{a+b}{a}$. 
By combining Lemma~\ref{lem:intEq2} and Theorem~\ref{thm:subdivision} its volume is also given by $\sum_{i=1}^w \mathrm{vol}\calQ_i$. 
The result then follows from $\calQ_i \equiv  \calF_{\widehat{G}(\nu(i))}$,
since $\vol \calF_{\widehat{G}(\nu(i))} = \Cat(\nu(i))$ by Lemma~\ref{lem:nuCarVol}.
\end{proof}

For any non-empty $k$-subset $S$ of $[w]$ the graph $\cap_{i\in S} \widehat{G}(\nu(i))$ has $k-1$ less edges than $\widehat{G}(\nu(i))$. 
Therefore, by the dimension formula for flow polytopes, $\calF_{\cap_{i\in S} \widehat{G}(\nu(i))} \equiv \cap_{i\in S}\calQ_i$ is a codimension $k-1$ inner face of the simplex-subdivision.
Since this holds for all non-empty $k$-subsets of $[w]$ where $1\leq k \leq w$, we have the following corollary.

\begin{corollary}
The dual to the subdivision $\cup_{i=1}^w \calQ_i$ of $\calF_{\widehat{G}_B(\nu)}$ is a $w$-simplex.
\end{corollary}

In light of the above corollary, we refer to the subdivision of Theorem~\ref{thm:subdivision} as the \textbf{simplex-subdivision} of $\calF_{\widehat{G}_B(\nu)}$. 
The simplex-subdivision of $\calF_{\widehat{G}_B(\nu)}$ then induces a subdivision of $\Delta_a\times \Delta_b$ via the integral equivalence in Lemma~\ref{lem:intEq2}.
We refer to this subdivision of $\Delta_a\times \Delta_b$ as the \textbf{simplex-subdivision induced by $\nu$}.

\begin{remark}
In the special case that $\nu = (NE)^n$, Theorem~\ref{thm:subdivision} is reminiscent of Cho's subdivision for the full root polytope of type $A_n$ in \cite[Theorem 16]{Cho99} (see also \cite[Theorem 7.1]{EHR18} where it is called the Legendre polytope). 
In fact, the simplex-subdivision of $\Delta_n\times\Delta_n$ induced by $\nu$ can be projected along the span of the vectors $\{(\be_i,\be_i)\mid i \in [n] \}$ to give Cho's subdivision of the full root polytope.   
\end{remark}

\section{The subdivision algebra}
\label{sec:4}

The subdivision algebra is an associative algebra given by M\'esz\'aros in which subdivisions of acyclic root polytopes \cite{Mes11} and certain flow polytopes \cite{Mes14} can be interpreted as reductions of monomials. 
In this section we use the simplex-subdivision from Theorem~\ref{thm:subdivision} to extend this tool of obtaining triangulations to a product of two simplices.
The graphs in this section will not have bidirectional edges, and so for convenience when $i<j$ we will write forward edges $(i,j,+)$ as $(i,j)$ and backward edges $(i,j,-)$ as $(j,i)$.

We begin with a quick overview of the subdivision algebra for flow polytopes. 
Let $G$ be a digraph with all forward edges, and let $(i,j)$ and $(j,k)$ be a pair of edges in $G$ with $i<j<k$. 
Define
\begin{align*}
    G_1 &:= (G \setminus \{(i,j)\}) \cup \{(i,k)\}, \\
    G_2 &:= (G \setminus \{(j,k)\}) \cup \{(i,k)\}, \text{ and} \\
    G_3 &:= (G \setminus \{(i,j),(j,k)\}) \cup \{(i,k,+)\}. 
\end{align*}

\begin{lemma}(Reduction lemma \cite[Lemma 2.2]{Mes14}, \cite[Proposition 2.3]{MS17})
Let $G$ be a digraph with only forward edges, and let $G_1$, $G_2$, and $G_3$ obtained from $G$ as above. Then there exists interior disjoint polytopes $\calP_1$ and $\calP_2$ such that $\calF_{\widetilde{G}} = \calP_1 \cup \calP_2$, where $\calP_1 \equiv \calF_{\widetilde{G}_1}$, $\calP_2\equiv \calF_{\widetilde{G}_2}$, and  $\calP_1\cap \calP_2 \equiv \calF_{\widetilde{G}_3}.$ 
\label{lem:reductionLemma}
\end{lemma}

The replacement of $G$ with the graphs $G_1$, $G_2$, and $G_3$ is called a \textbf{reduction}.
A reduction of $G$ can be viewed as a rooted tree, with $G$ as the root with the graphs $G_1$, $G_2$, and $G_3$ as leaves.
Continuing to perform reductions on the leaves until no more reductions are possible, we obtain a rooted ternary tree called a \textbf{reduction tree}. 
It is important to note here that when a reduction is performed on a leaf at a pair $\{(i,j),(j,k)\}$, the same pair is reduced in all other leaves where such a reduction is possible.
A reduction tree obtained in this way is not unique in general, and depends on the reduction order. 
However, the number of leaves in two reduction trees of a graph is always the same \cite[Lemma 1]{Mes14}. 
In light of Lemma~\ref{lem:reductionLemma}, the reduction tree of a graph $G$ encodes a triangulation of the flow polytope $\calF_{\widetilde{G}}$, as is made explicit in the following.

\begin{proposition}(\cite[Theorem 3]{Mes14})
Let $G$ be a digraph with only forward edges. 
Then the leaves of a reduction tree of $G$ encode a regular flag triangulation of $\calF_{\widetilde{G}}$.
Each leaf with $\abs{E(G)}-k$ edges encodes a codimension $k$ inner face of the triangulation. \qed
\label{prop:redTree}
\end{proposition}

If $L$ is a leaf in a reduction tree of $G$, then we recover the corresponding inner face (up to an integral equivalence) of the triangulation of $\calF_{\widetilde{G}}$ by computing $\calF_{\widetilde{L}}$. 
If we restrict to a graph $G(\nu)$, then Proposition~\ref{prop:redTree} can be rephrased to holds for the flow polytope of the $\nu$-augmentation of $G(\nu)$ as follows.

\begin{corollary}
The leaves of a reduction tree of $G(\nu)$ encode a regular flag triangulation of $\calF_{\widehat{G}(\nu)}$. 
Each leaf with $\abs{E(G(\nu))}-k$ edges encodes a codimension $k$ inner face of the triangulation. 
\label{cor:nuAug}
\end{corollary}

\begin{proof}
Let $\scrT$ be a reduction tree of $G(\nu)$ and let $\calT$ be the flag regular triangulation of $\calF_{\widetilde{G}(\nu)}$ it encodes by Proposition~\ref{prop:redTree}. 
Observe that a $\nu$-augmented graph $\widehat{G}(\nu)$ can be obtained from its full augmentation $\widetilde{G}(\nu)$ by removing the edges $\calE_s$ of the form $(s,j)$ where $j\in J\setminus I$ and the set of edges $\calE_t$  of the form $(i,t) $ with $i\in I\setminus J$. 
A vertex $i \in I\setminus J$ is a source of $G(\nu)$ and a vertex $j \in J\setminus I$ is a sink. 
Since reductions do not remove sources or sinks, these vertices are also sources and sink in each leaf of $\scrT$. 
Therefore, removing the edges in $\calE_s \cup \calE_t$ amounts to removing precisely the routes of the form $\{(s,i),(i,t)\}$ with  $i \notin I\cap J$ from $\widetilde{G}(\nu)$ and $\widetilde{L}$ for any leaf $L$ of $\scrT$.
It follows that $\calF_{\widehat{G}(\nu)}$ is a codimension $\abs{\calE_s\cup \calE_t}$ face of $\calF_{\widetilde{G}(\nu)}$. 
The restriction of $\calT$ to the face $\calF_{\widehat{G}(\nu)}$ is now a regular flag triangulation encoded by the leaves of $\scrT$.
The encoding is the same as in the fully augmented case, with the understanding that the inner face corresponding to a leaf $L$ of $\scrT$ is now given by the flow polytope of $\widehat{L}$ instead of $\widetilde{L}$.
\end{proof}

The geometry of reductions can now be encoded in an algebraic setting known as the subdivision algebra. 
In particular, a reduction can be encoded using the following \textbf{reduction relation} $ x_{ij}x_{jk} = x_{ij}x_{ik} + x_{ik}x_{jk} + \beta x_{ik}.$
See Figure~\ref{fig:reduction} for a visualization relating the terms in the reduction relation to graphs. 
In particular, as a consequence of Lemma~\ref{lem:reductionLemma}, a reduction of the pair $x_{ij}x_{jk}$ with $i<j<k$ encodes the cutting of the flow polytope by the hyperplane $\{\bx\in\bbR^m \mid x_{(i,j)} = x_{(j,k)}\}$. 
The subdivision algebra is defined as follows.

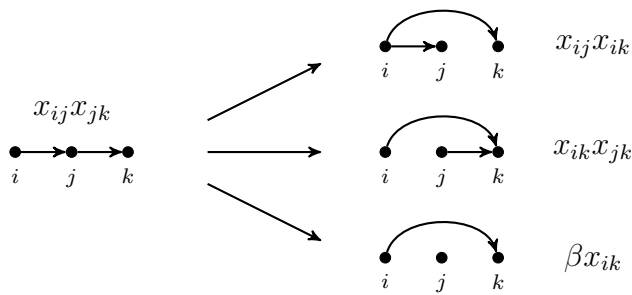
\begin{figure}
\begin{center}
\begin{tikzpicture}
\begin{scope}[xshift=0, yshift=0, scale=0.75]
	\vertex[fill,label=below:\tiny{$i$}](a1) at (1,0) {};
	\vertex[fill,label=below:\tiny{$j$}](a2) at (2,0) {};
	\vertex[fill,label=below:\tiny{$k$}](a3) at (3,0) {};

	\draw[>=stealth', ->, thick] (a1)--(a2);
	\draw[>=stealth', ->, thick] (a2)--(a3);

    \node[] (l1) at (0,0.5) {};
    \node[] (l2) at (0,1.2) {};
    \node[] (r1) at (4,0.5) {};
    \node[] (r2) at (4,1.2) {};    
\end{scope}

\begin{scope}[xshift=140, yshift=40, scale=0.75]
	\vertex[fill,label=below:\tiny{$i$}](a1) at (1,0) {};
	\vertex[fill,label=below:\tiny{$j$}](a2) at (2,0) {};
	\vertex[fill,label=below:\tiny{$k$}](a3) at (3,0) {};

	\draw[>=stealth', ->, thick] (a1)--(a2);
    \draw[>=stealth', ->, thick] (a1) .. controls (1.25, 0.8) and (2.75, 0.8) .. (a3);

\end{scope}

\begin{scope}[xshift=140, yshift=0, scale=0.75]
	\vertex[fill,label=below:\tiny{$i$}](a1) at (1,0) {};
	\vertex[fill,label=below:\tiny{$j$}](a2) at (2,0) {};
	\vertex[fill,label=below:\tiny{$k$}](a3) at (3,0) {};

	\draw[>=stealth', ->, thick] (a2)--(a3);
   \draw[>=stealth', ->, thick] (a1) .. controls (1.25, 0.8) and (2.75, 0.8) .. (a3);

\end{scope}

\begin{scope}[xshift=140, yshift=-40, scale=0.75]
	\vertex[fill,label=below:\tiny{$i$}](a1) at (1,0) {};
	\vertex[fill,label=below:\tiny{$j$}](a2) at (2,0) {};
	\vertex[fill,label=below:\tiny{$k$}](a3) at (3,0) {};

   \draw[>=stealth', ->, thick] (a1) .. controls (1.25, 0.8) and (2.75, 0.8) .. (a3);

\end{scope}

\begin{scope}[xshift=90, yshift=-10, scale=0.9]
\node[] (a1) at (0,0) {};
\node[] (a2) at (2,-1) {};
\draw[>=stealth', ->, thick] (a1)--(a2);
\end{scope}

\begin{scope}[xshift=90, yshift=10, scale=0.9]
\node[] (a1) at (0,0) {};
\node[] (a2) at (2,1) {};
\draw[>=stealth', ->, thick] (a1)--(a2);
\end{scope}

\begin{scope}[xshift=90, yshift=0, scale=0.9]
\node[] (a1) at (0,0) {};
\node[] (a2) at (2,0) {};
\draw[>=stealth', ->, thick] (a1)--(a2);
\end{scope}

\begin{scope}[xshift=43, yshift=15, scale=0.9]
\node[] (a1) at (0,0) {$x_{ij}x_{jk}$};
\end{scope}

\begin{scope}[xshift=240, yshift=40, scale=0.9]
\node[] (a1) at (0,0) {$x_{ij}x_{ik}$};
\end{scope}

\begin{scope}[xshift=240, yshift=0, scale=0.9]
\node[] (a1) at (0,0) {$x_{ik}x_{jk}$};
\end{scope}

\begin{scope}[xshift=240, yshift=-40, scale=0.9]
\node[] (a1) at (0,0) {$\beta x_{ik}$};
\end{scope}

\end{tikzpicture}
\end{center}
    \caption{The graphical representation of a reduction and the associated terms in the subdivision algebra}
    \label{fig:reduction}
\end{figure}

\begin{definition} 
\label{def:subdivAlgebra}
The \textbf{subdivision algebra} $\calS(\beta)$ is an associative and commutative algebra over the ring of polynomials $\bbZ[\beta]$ generated by $\{x_{ij} \mid i\in [n], j\in [n], i\neq j\}$, subject to the reduction relation
$x_{ij}x_{jk} = x_{ik}x_{ij} + x_{jk}x_{ik} + \beta x_{ik}$, where $i\neq k$. 
For fixed $i$, $j$, and $k$, replacing all terms $x_{ij}x_{jk}$ in a polynomial $P$ in $\calS(\beta)$ with $x_{ik}x_{ij} + x_{jk}x_{ik} + \beta x_{ik}$ is called a \textbf{reduction}. 
If applying a sequence of reductions to $P$ yields a polynomial $R$ that cannot be reduced further, we say that $R$ is a \textbf{reduced form} of $P$.   
\end{definition}

The definition above is slightly more general than the one given by M\'esz\'aros \cite{Mes11} in that here a generator $x_{ij}$ need not satisfy $i<j$, and we do not require that $i < j < k$ in the reduction relation. 
Figure~\ref{fig:reduction} gives a visualization of a reduction.
Sometimes it is more convenient to consider \textbf{simple reductions}, which only encode the full-dimensional polytopes in a subdivision. 
These are obtained by setting $\beta=0$ in the reduction relation. 
Let us consider a small example. 

\begin{example}
Let $P = x_{12}x_{24}x_{52}$. We can apply reductions as follows. 
\begin{align*}
    \mathbf{x_{12}x_{24}}x_{52} &= x_{12}x_{14}x_{52} + x_{14}\mathbf{x_{24}x_{52}} + x_{14}x_{52}\beta \\
    &= x_{12}x_{14}x_{52} + x_{14}x_{52}x_{54} + x_{14}x_{24}x_{54} + x_{14}x_{54}\beta + x_{14}x_{52}\beta \\
\end{align*}
The pairs in bold indicate the pairs that were reduced at each step.
In the second reduction, we used the commutativity to write $\bx_{24} \bx_{52}$ as $\bx_{52}\bx_{24}$ in order to perform the reduction.
Since no more reductions are possible in the last line, the resulting polynomial is a reduced form of $P$.
Setting $\beta=0$ gives the simple reductions at each step.
\end{example}

Given a digraph $G$ with forward and backward edges, we identify an edge $(i,j)$ the generator $x_{ij}$. 
The graph $G$ can then be identified with the monomial $M(G) := \prod_{(i,j)\in E(G)} x_{ij}$.

\begin{theorem}
\label{thm:polynomial}
The simplex-subdivision of $\calF_{\widehat{G}_B(\nu)}$ is encoded in the polynomial 
$$P_\nu := \sum_{\substack{S \subseteq [w] \\ S\neq \emptyset}} \beta^{\, \abs{S}-1} M(\cap_{i\in S}G(\nu(i))) $$
where $w$ is the number of valleys in $\overline{\nu} = E\nu N$. 
Furthermore, reduced forms of $P_\nu$ encode triangulations of $\calF_{\widehat{G}_B(\nu)}$. 
\end{theorem}

\begin{proof}
By Theorem~\ref{thm:subdivision}, each nonempty $S\subseteq [w]$ corresponds bijectively to a flow polytope $\calF_{\cap_{i\in S}\widehat{G}(\nu(i))}$ that is integrally equivalent to an inner face of the simplex-subdivision of $\calF_{\widehat{G}_B(\nu)}$.
We associate this flow polytope with the term $\beta^{\,\abs{S}-1}M(\cap_{i\in S}G(\nu(i)))$, where $\beta^{\,\abs{S}-1}$ encodes the fact that it has codimension $\abs{S}-1$ and the monomial $M(\cap_{i\in S}G(\nu(i)))$ records the inner edges of the graph $\cap_{i\in S}\widehat{G}(\nu(i))$. 
Summing all such terms over the nonempty subsets of $[w]$ gives the polynomial $P_\nu$, whose terms thereby encode the flow polytopes corresponding to the inner faces of the simplex-subdivision of $\calF_{\widehat{G}_B(\nu)}$. 

To see that reduced forms of $P_\nu$ encode triangulations of $\calF_{\widehat{G}_B(\nu)}$, consider first an inner face of the simplex-subdivision corresponding to a flow polytope $\calF_{H}$ for some $H:= \cap_{i\in S}\widehat{G}(\nu(i))$ with nonempty $S\subseteq [w]$.
The graph $H$ has at most one backward edge, and is a subgraph of $\widehat{G}(\nu(\ell))$ for some $\ell \in S$. 
We can apply the cyclic relabeling $\sigma_\ell$ from Lemma~\ref{lem:sigma} on the inner vertices of $H$ to obtain a graph $H'$ where all of its edges are forward edges.
See Figure~\ref{fig:reordered} for an example of such a relabeling.
Let $H_*$ and $H'_*$ denote the subgraphs induced by the inner edges of $H$ and $H'$ respectively.
By Proposition~\ref{prop:redTree}, the leaves of a reduction tree $\scrT'$ for the graph $H'_*$ encode a flag regular triangulation of the flow polytope $\calF_{H'}$.
Applying $\sigma^{-1}_\ell$ to the vertices of the graphs in $\scrT'$ gives a reduction tree $\scrT$ for $H_*$. 
Changing the sign of the flow in the edge $(v_w,v_1)$ is an integral equivalence (as seen in Lemma~\ref{lem:intEq}), so the flow polytopes of the leaves of $\scrT'$ are integrally equivalent to those from $\scrT$.
Thus the leaves of $\scrT$ encode a regular unimodular triangulation of $\calF_H$.

Since any inner face of the simplex-subdivision can now be subdivided by reducing the corresponding monomial in $P_\nu$, we obtain subdivisions of $\calF_{\widehat{G}_B(\nu)}$ when performing reductions on $P_\nu$. 
\end{proof}

\begin{figure}
\begin{center}
\begin{tikzpicture}
\begin{scope}[xshift=-20, yshift=0, scale=0.8]
\vertex[fill,label=below:\scriptsize{$s$}](a0) at (0,0) {};
	\vertex[fill,label=below:\tiny{$1$}](a1) at (1,0) {};
	\vertex[fill,label=below:\tiny{$2$}](a2) at (2,0) {};
	\vertex[fill,label=below:\tiny{$3$}](a3) at (3,0) {};
	\vertex[fill,label=below:\tiny{$4$}](a4) at (4,0) {};
	\vertex[fill,label=below:\tiny{$5$}](a5) at (5,0) {};
	\vertex[fill,label=below:\scriptsize{$t$}](a6) at (6,0) {};

	\draw[>=stealth', ->, color=gray] (a0)--(a1);	
	\draw[>=stealth', ->, very thick] (a1)--(a2);
	\draw[>=stealth', ->, very thick] (a3)--(a4);
	\draw[>=stealth', ->, very thick] (a4)--(a5);
	\draw[>=stealth', <-, very thick] (a1) .. controls (1.25, 1.6) and (4.75, 1.6) .. (a5);
	\draw[>=stealth', ->, color=gray] (a5)--(a6);
	\draw[>=stealth', ->, color=gray] (a0) .. controls (0.25, 0.7) and (1.75, 0.7) .. (a2);
	\draw[>=stealth', ->, color=gray] (a0) .. controls (0.25, 1.4) and (2.75, 1.4) .. (a3);
	\draw[>=stealth', ->, color=gray] (a0) .. controls (0.25, 2.1) and (3.75, 2.1) .. (a4);
	\draw[>=stealth', ->, color=gray] (a0) .. controls (0.25, 2.8) and (4.75, 2.8) .. (a5);
	\draw[>=stealth', ->, color=gray] (a4) .. controls (4.25, 0.7) and (5.75, 0.7) .. (a6);
	\draw[>=stealth', ->, color=gray] (a2) .. controls (2.25, 2.1) and (5.75, 2.1) ..  (a6);
	\draw[>=stealth', ->, color=gray] (a1) .. controls (1.25, 2.8) and (5.75, 2.8) ..  (a6);
\end{scope}

\begin{scope}[xshift=180, yshift=0, scale=0.8]
\vertex[fill,label=below:\scriptsize{$s$}](a0) at (0,0) {};
	\vertex[fill,label=below:\tiny{$1$}](a1) at (1,0) {};
	\vertex[fill,label=below:\tiny{$2$}](a2) at (2,0) {};
	\vertex[fill,label=below:\tiny{$3$}](a3) at (3,0) {};
	\vertex[fill,label=below:\tiny{$4$}](a4) at (4,0) {};
	\vertex[fill,label=below:\tiny{$5$}](a5) at (5,0) {};
	\vertex[fill,label=below:\scriptsize{$t$}](a6) at (6,0) {};

	\draw[>=stealth', ->, color=gray] (a0)--(a1);	
	\draw[>=stealth', ->, very thick] (a1)--(a2);
	\draw[>=stealth', ->, very thick] (a2)--(a3);
	\draw[>=stealth', ->, very thick] (a3)--(a4);
	\draw[>=stealth', ->, very thick] (a4)--(a5);

	\draw[>=stealth', ->, color=gray] (a5)--(a6);
	\draw[>=stealth', ->, color=gray] (a0) .. controls (0.25, 0.7) and (1.75, 0.7) .. (a2);
	\draw[>=stealth', ->, color=gray] (a0) .. controls (0.25, 1.4) and (2.75, 1.4) .. (a3);
	\draw[>=stealth', ->, color=gray] (a0) .. controls (0.25, 2.1) and (3.75, 2.1) .. (a4);
	\draw[>=stealth', ->, color=gray] (a0) .. controls (0.25, 2.8) and (4.75, 2.8) .. (a5);
	\draw[>=stealth', ->, color=gray] (a4) .. controls (4.25, 0.7) and (5.75, 0.7) .. (a6);
	\draw[>=stealth', ->, color=gray] (a2) .. controls (2.25, 2.1) and (5.75, 2.1) ..  (a6);
	\draw[>=stealth', ->, color=gray] (a1) .. controls (1.25, 2.8) and (5.75, 2.8) ..  (a6);
\end{scope}

\begin{scope}[xshift=145, yshift=15, scale=0.8]
\node[] (a) at (0,0) {$\longrightarrow$};
\end{scope}

\begin{scope}[xshift=50, yshift=-30, scale=0.8]
\node[] (a) at (0,0) {$x_{34}x_{45}x_{51}x_{12}$};
\end{scope}
\begin{scope}[xshift=250, yshift=-30, scale=0.8]
\node[] (a) at (0,0) {$x_{12}x_{23}x_{34}x_{45}$};
\end{scope}
\end{tikzpicture}
\end{center}
    \caption{A cyclic relabeling of the inner vertices of $\widehat{G}(\nu(2))$ from Figure~4 by $\sigma_2$ gives a graph with only forward edges}
    \label{fig:reordered}
\end{figure}
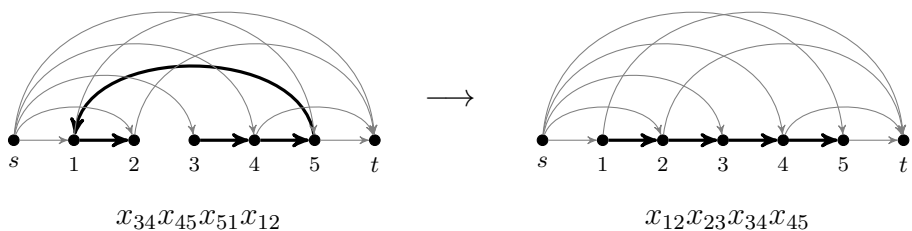

By the integral equivalence in Lemma~\ref{lem:intEq2} we then have the following. 

\begin{corollary}
Let $\nu$ be a lattice path from $(0,0)$ to $(a,b)$. 
Then reducing $P_\nu$ encodes subdivisions of $\Delta_a\times \Delta_b$.
\end{corollary}

The triangulations obtained via Theorem~\ref{thm:polynomial} are unimodular, as all triangulations of a product of simplices are unimodular \cite[Proposition 6.2.11]{DRS10}. 
Triangulating each $\calQ_i$ with the subdivision algebra gives regular triangulations by Corollary~\ref{cor:nuAug}.
Thus the triangulation encoded by a reduced form of $P_\nu$ is a union of $w$ regularly triangulated polytopes.
It would be interesting to know if the triangulation formed by their union is also regular.

We conclude this section with an example of applying the subdivision algebra directly to a product of two simplices $\Delta_{a} \times \Delta_{b}$.

\begin{example}
\label{ex.a3b2}
Let $a = 3$ and $b=2$. To subdivide $\Delta_{3} \times \Delta_{2}$, we first choose a path $\nu$ from $(0,0)$ to $(3,2)$, say $\nu = NEENE$, and consider $\overline{\nu}$ with its canonical indexing $E_1N_1E_2E_3N_3E_4N_4$.
Now $I = \{1,2,3,4\}$ and $J =\{1,3,4\}$, and $I\cap J = \{1,3,4\}$. 
Cycling the path $\overline{\nu}$ to start at its cyclic peaks we have $\overline{\nu}(1) = E_2E_3N_3E_4N_4E_1N_1$, $\overline{\nu}(2) = E_4N_4E_1N_1E_2E_3N_3$, and $\overline{\nu}(3) = E_1N_1E_2E_3N_3E_4N_4$. 
From the resulting graphs $G(\nu(i))$ for $i\in[3]$ we read the monomials
$M_1 = x_{23}x_{34}x_{41}$,
$M_2 = x_{41}x_{13}x_{23}$,
and $M_3 = x_{13}x_{23}x_{34}$.
The polynomial $P_\nu$ is now given by 
$$P_\nu = \sum_{\substack{S \subseteq [3] \\ S\neq \emptyset}} \beta^{\,\abs{S}-1} \gcd\{M_i\mid i \in S\}, $$
which expands to
$$ P_\nu = x_{23}x_{34}x_{41} + x_{41}x_{13}x_{23} + x_{13}x_{23}x_{34} + x_{23}x_{41}\beta + x_{23}x_{34}\beta + x_{13}x_{23}\beta + x_{23}\beta^2.
$$

\begin{figure}
\begin{center}
\begin{tikzpicture}
\begin{scope}[xshift = 0, yshift = 0, scale=0.6]
	\node[style={circle,draw, inner sep=2pt, fill=black}] (p1)  at (0,0)  {};
	\node[style={circle,draw, inner sep=2pt, fill=black}] (p2)  at (2,1)  {};
    \node[style={circle,draw, inner sep=2pt, fill=black}] (p3)  at (-2,1)  {}; 
    \node[style={circle,draw, inner sep=2pt, fill=black}] (p4)  at (-1,2.5)  {};    
    \node[style={circle,draw, inner sep=2pt, fill=black}] (p5)  at (1,2.5)  {};
    
    \draw[ultra thick, red] (p1) to (p2);
    \draw[ultra thick, red] (p2) to (p5);
    \draw[ultra thick, red] (p5) to (p4);    
    \draw[ultra thick, red] (p4) to (p3);
    \draw[ultra thick, red] (p1) to (p3);    
\end{scope}

\begin{scope}[xshift = 33, yshift = -50, scale=0.6]
	\node[style={circle,draw, inner sep=2pt, fill=black}] (q1)  at (0,0)  {};
	\node[style={circle,draw, inner sep=2pt, fill=black}] (q2)  at (1.5,1.5)  {};
    \node[style={circle,draw, inner sep=2pt, fill=black}] (q0)  at (-1.8,-1)  {}; 
    
    \draw[ultra thick, red] (q1) to (q2);
    \draw[ultra thick, red] (q0) to (q1);    
\end{scope}

\begin{scope}[xshift = -42, yshift = -56, scale=0.6]
	\node[style={circle,draw, inner sep=2pt, fill=black}] (r1)  at (0,0)  {};
	\node[style={circle,draw, inner sep=2pt, fill=black}] (r2)  at (-1,2)  {};
    
    \draw[ultra thick, red] (r1) to (r2);
\end{scope}
    \draw[] (q1) to (p1);
    \draw[] (q2) to (p2);
    \draw[] (r1) to (q0);
    \draw[] (r2) to (p3);
    
	\node[] (x1)  at (-1.5,1.8)  {$x_{21}x_{31}x_{41}$};    
	\node[] (x1)  at (1.5,1.8)  {$x_{21}x_{31}x_{34}$};
	\node[] (x1)  at (-2.2,0.8)  {$x_{21}x_{23}x_{41}$};   
	\node[] (x1)  at (2.2,0.8)  {$x_{21}x_{24}x_{34}$};
	\node[] (x1)  at (-0.65,-0.35)  {$x_{21}x_{23}x_{24}$};  
	
	\node[] (x1)  at (3.1,-0.7)  {$x_{14}x_{24}x_{34}$};  	
	\node[] (x1)  at (2.2,-1.8)  {$x_{14}x_{23}x_{24}$};
	\node[] (x1)  at (0,-2.7)  {$x_{13}x_{14}x_{23}$};	
	
	\node[] (x1)  at (-3.1,-0.7)  {$x_{23}x_{41}x_{43}$};  	
	\node[] (x1)  at (-2.55,-1.85)  {$x_{13}x_{23}x_{43}$};  	
\end{tikzpicture}
\end{center}
    \caption{The dual graph of the triangulation obtained in Example~\ref{ex.a3b2}}
    \label{fig:dualGraph}
\end{figure}
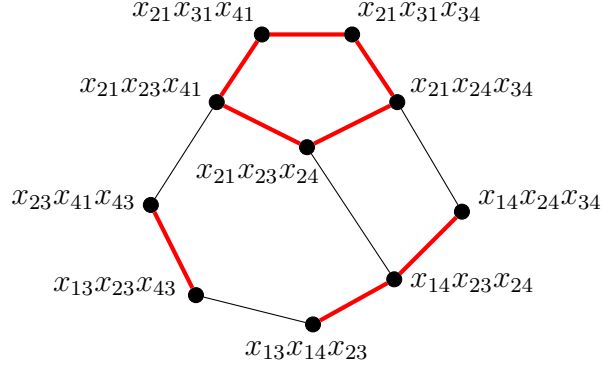

For simplicity, we consider only simple reductions and therefore set $\beta=0$. We then reduce this polynomial as follows. 

\begin{align*}
    P_\nu(\beta=0) &= x_{23}x_{34}x_{41} + \mathbf{x_{13}}x_{23}\mathbf{x_{41}} + x_{13}x_{23}x_{34} \\
    &= x_{23}x_{34}x_{41} + x_{13}x_{23}x_{43} + x_{23}x_{41}x_{43} + \mathbf{x_{13}}x_{23}\mathbf{x_{34}} \\
    &= \mathbf{x_{23}x_{34}}x_{41} + x_{13}x_{23}x_{43} + x_{23}x_{41}x_{43} + x_{13}x_{14}x_{23} + x_{14}\mathbf{x_{23}x_{34}} \\
    &= x_{23}\mathbf{x_{24}x_{41}} + \mathbf{x_{24}}x_{34}\mathbf{x_{41}} + x_{13}x_{23}x_{43} + x_{23}x_{41}x_{43} + x_{13}x_{14}x_{23} + \\ 
    & \qquad  x_{14}x_{23}x_{24} + x_{14}x_{24}x_{34} \\
    &= x_{23}x_{24}x_{21} + x_{23}x_{21}x_{41} + x_{24}x_{34}x_{21} + x_{21}\mathbf{x_{34}x_{41}} + x_{13}x_{23}x_{43} +  \\ 
    & \qquad x_{23}x_{41}x_{43} + x_{13}x_{14}x_{23} + x_{14}x_{23}x_{24} + x_{14}x_{24}x_{34} \\
    &= x_{23}x_{24}x_{21} + x_{23}x_{21}x_{41} + x_{24}x_{34}x_{21} + x_{21}x_{34}x_{31} + x_{21}x_{31}x_{41} +  \\ 
    & \qquad x_{13}x_{23}x_{43} + x_{23}x_{41}x_{43} + x_{13}x_{14}x_{23} + x_{14}x_{23}x_{24} + x_{14}x_{24}x_{34} \\    
\end{align*}

Each summand in the reduced form corresponds to a simplex, with each generator $x_{ij}$ giving the vertex $(\be_i,\be_j)$ of the simplex, where $\be_i$ and $\be_j$ are the $i$th and $j$th standard basis vectors in $\bbR^4$ and $\bbR^3$ respectively. 
For example, the simplex corresponding to the summand $x_{23}x_{24}x_{21}$ is the convex hull of vertices  $(\be_2,\be_3)$, $(\be_2,\be_4)$, and $(\be_2,\be_1)$, along with the cone points of the triangulation, which are $(\be_k,\be_k)$ where $k\in I\cap J = \{1,3,4\}$. From the reduced form we also obtain the dual graph of the triangulation, with an edge between monomials differing by a single generator. The dual graph is shown in Figure~\ref{fig:dualGraph}.
\end{example}

\section{The $\nu$-cyclohedral triangulations of $\Delta_a\times\Delta_b$.}
\label{sec:5}

In this section, we use the subdivision algebra to obtain a combinatorially interesting triangulation of $\calF_{\widehat{G}_B(\nu)}$, or equivalently, a product of two simplices. 
The triangulation in question is a geometric realization of the cyclic $\nu$-Tamari complex of \cite{CPS19} and we refer to it as the $\nu$-cyclohedral triangulation of $\Delta_a\times \Delta_b$. 
The $\nu$-Tamari complex and cyclic $\nu$-Tamari complex were introduced by Ceballos, Padrol, and Sarmiento in \cite{CPS19} in their study of the geomtry of Tamari lattices. 
We will recap the necessary definitions and results below.

Given a lattice path $\nu$ and the canonically indexed $\overline{\nu}$ using labels $\{1,\ldots, n\}$, with $I$ and $J$ being the index sets of $E$ steps and $N$ steps respectively.
A pair $(i,j)$ with $i\in I$ and $j\in J$ is said to be an \textbf{arc}. 
If the arc $(i,j)$ satisfies $i\leq j$, we say that it is an \textbf{increasing arc}.
We denote the set of all arcs on $\overline{\nu}$ by $\calA_{\overline{\nu}}$. Define the length of an arc $(i,j)$ to be $j-i \pmod n$
An arc $(i,j)$ is \textbf{minimal} if $i=j$, and it is \textbf{maximal} if $i=1$ and $j=n$, or $j<i$ and there is no $k\in I\cup J$ with $j < k < i$.
Maximal arcs have length $n-1$. 
Two arcs $(i,j)$ and $(i',j')$ are said to \textbf{cyclically cross} if any of the following conditions hold (up to reversing the roles of the arcs): 
\begin{align*}
&(1) \quad i < i' < j < j', &(2) \quad j' < i < i' < j,\quad &\quad \quad (3) \quad j < j' < i < i', \\ 
&(4) \quad i' < j < j' < i, &(5) \quad i < j' < i' < j,\quad &\quad \quad(6) \quad j < i < j' < i'.
\end{align*} 

These are visualized in Figure~\ref{fig:cycCrossing}.

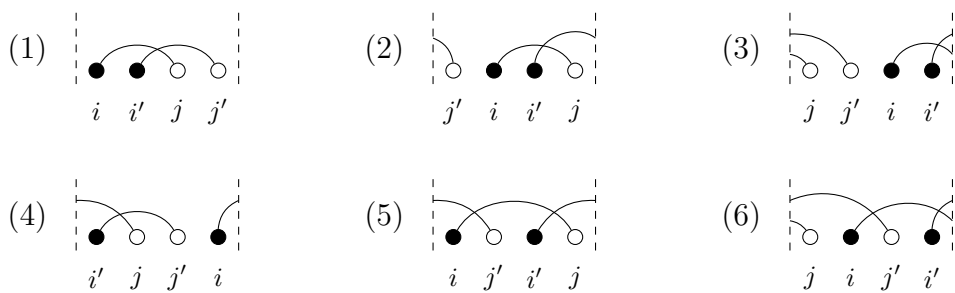
\begin{figure}[ht!]
\begin{center}
\begin{tikzpicture}[scale=0.9]
\begin{scope}[xshift = 0, yshift = 0, scale=0.6]
	\node[style={circle,draw, inner sep=2pt, fill=black}] (1)  at (1,0)  {};
	\node[style={circle,draw, inner sep=2pt, fill=black}] (2)  at (2,0)  {};
	\node[style={circle,draw, inner sep=2pt, fill=none}] (3)  at (3,0)  {};
	\node[style={circle,draw, inner sep=2pt, fill=none}] (4)  at (4,0)  {};
	
	\node[] (a)  at (1,-1)  {{\footnotesize $i$}};
	\node[] (b)  at (2,-1)  {{\footnotesize $i'$}};
	\node[] (c)  at (3,-1)  {{\footnotesize $j$}};
   \node[] (d)  at (4,-1) {{\footnotesize $j'$}};	
	
	\draw[] (1) to [bend left=60] (3);
   \draw[] (2) to [bend left=60] (4);
   
   \node[] (w1) at (0.5,-0.6) {};
   \node[] (w2) at (0.5,1.8) {};
   \node[] (w3) at (4.5,-0.6) {};
   \node[] (w4) at (4.5,1.8) {};
	\draw[dashed] (w1) -- (w2);  
	\draw[dashed] (w3) -- (w4);  	
	
	\node[] (l) at (-0.7,0.5) {(1)};
\end{scope}

\begin{scope}[xshift = 150, yshift = 0, scale=0.6]
	\node[style={circle,draw, inner sep=2pt, fill=none}] (1)  at (1,0)  {};
	\node[style={circle,draw, inner sep=2pt, fill=black}] (2)  at (2,0)  {};
	\node[style={circle,draw, inner sep=2pt, fill=black}] (3)  at (3,0)  {};
	\node[style={circle,draw, inner sep=2pt, fill=none}] (4)  at (4,0)  {};
	
	\node[] (a)  at (1,-1)  {{\footnotesize $j'$}};
	\node[] (b)  at (2,-1)  {{\footnotesize $i$}};
	\node[] (c)  at (3,-1)  {{\footnotesize $i'$}};
   \node[] (d)  at (4,-1) {{\footnotesize $j$}};	
	
	\draw[] (3) to [bend left=60] (4.5,0.8);
	\draw[] (0.5,0.8) to [bend left=30] (1);
   \draw[] (2) to [bend left=60] (4);
   
   \node[] (w1) at (0.5,-0.6) {};
   \node[] (w2) at (0.5,1.8) {};
   \node[] (w3) at (4.5,-0.6) {};
   \node[] (w4) at (4.5,1.8) {};
	\draw[dashed] (w1) -- (w2);  
	\draw[dashed] (w3) -- (w4);  	

	\node[] (l) at (-0.7,0.5) {(2)};
\end{scope}

\begin{scope}[xshift = 300, yshift = 0, scale=0.6]
	\node[style={circle,draw, inner sep=2pt, fill=none}] (1)  at (1,0)  {};
	\node[style={circle,draw, inner sep=2pt, fill=none}] (2)  at (2,0)  {};
	\node[style={circle,draw, inner sep=2pt, fill=black}] (3)  at (3,0)  {};
	\node[style={circle,draw, inner sep=2pt, fill=black}] (4)  at (4,0)  {};
	
	\node[] (a)  at (1,-1)  {{\footnotesize $j$}};
	\node[] (b)  at (2,-1)  {{\footnotesize $j'$}};
	\node[] (c)  at (3,-1)  {{\footnotesize $i$}};
   \node[] (d)  at (4,-1) {{\footnotesize $i'$}};	
	
	\draw[] (3) to [bend left=60] (4.5,0.4);
	\draw[] (0.5,0.4) to [bend left=20] (1);
	\draw[] (0.5,0.9) to [bend left=30] (2);
   \draw[] (4) to [bend left=30] (4.5,0.9);

   \node[] (w1) at (0.5,-0.6) {};
   \node[] (w2) at (0.5,1.8) {};
   \node[] (w3) at (4.5,-0.6) {};
   \node[] (w4) at (4.5,1.8) {};
	\draw[dashed] (w1) -- (w2);  
	\draw[dashed] (w3) -- (w4);  
	
	\node[] (l) at (-0.7,0.5) {(3)};		
\end{scope}

\begin{scope}[xshift = 0, yshift = -70, scale=0.6]
	\node[style={circle,draw, inner sep=2pt, fill=black}] (1)  at (1,0)  {};
	\node[style={circle,draw, inner sep=2pt, fill=none}] (2)  at (2,0)  {};
	\node[style={circle,draw, inner sep=2pt, fill=none}] (3)  at (3,0)  {};
	\node[style={circle,draw, inner sep=2pt, fill=black}] (4)  at (4,0)  {};
	
	\node[] (a)  at (1,-1)  {{\footnotesize $i'$}};
	\node[] (b)  at (2,-1)  {{\footnotesize $j$}};
	\node[] (c)  at (3,-1)  {{\footnotesize $j'$}};
   \node[] (d)  at (4,-1) {{\footnotesize $i$}};	
	
	\draw[] (1) to [bend left=60] (3);
	\draw[] (0.5,0.9) to [bend left=30] (2);
   \draw[] (4) to [bend left=30] (4.5,0.9);
   
   \node[] (w1) at (0.5,-0.6) {};
   \node[] (w2) at (0.5,1.8) {};
   \node[] (w3) at (4.5,-0.6) {};
   \node[] (w4) at (4.5,1.8) {};
	\draw[dashed] (w1) -- (w2);  
	\draw[dashed] (w3) -- (w4);  
	
	\node[] (l) at (-0.7,0.5) {(4)};		
\end{scope}

\begin{scope}[xshift = 150, yshift = -70, scale=0.6]
	\node[style={circle,draw, inner sep=2pt, fill=black}] (1)  at (1,0)  {};
	\node[style={circle,draw, inner sep=2pt, fill=none}] (2)  at (2,0)  {};
	\node[style={circle,draw, inner sep=2pt, fill=black}] (3)  at (3,0)  {};
	\node[style={circle,draw, inner sep=2pt, fill=none}] (4)  at (4,0)  {};
	
	\node[] (a)  at (1,-1)  {{\footnotesize $i$}};
	\node[] (b)  at (2,-1)  {{\footnotesize $j'$}};
	\node[] (c)  at (3,-1)  {{\footnotesize $i'$}};
   \node[] (d)  at (4,-1) {{\footnotesize $j$}};	
	
	\draw[] (1) to [bend left=60] (4);
	\draw[] (0.5,0.9) to [bend left=30] (2);
   \draw[] (3) to [bend left=30] (4.5,0.9);
   
   \node[] (w1) at (0.5,-0.6) {};
   \node[] (w2) at (0.5,1.8) {};
   \node[] (w3) at (4.5,-0.6) {};
   \node[] (w4) at (4.5,1.8) {};
	\draw[dashed] (w1) -- (w2);  
	\draw[dashed] (w3) -- (w4);  	
	
	\node[] (l) at (-0.7,0.5) {(5)};
\end{scope}

\begin{scope}[xshift = 300, yshift = -70, scale=0.6]
	\node[style={circle,draw, inner sep=2pt, fill=none}] (1)  at (1,0)  {};
	\node[style={circle,draw, inner sep=2pt, fill=black}] (2)  at (2,0)  {};
	\node[style={circle,draw, inner sep=2pt, fill=none}] (3)  at (3,0)  {};
	\node[style={circle,draw, inner sep=2pt, fill=black}] (4)  at (4,0)  {};
	
	\node[] (a)  at (1,-1)  {{\footnotesize $j$}};
	\node[] (b)  at (2,-1)  {{\footnotesize $i$}};
	\node[] (c)  at (3,-1)  {{\footnotesize $j'$}};
   \node[] (d)  at (4,-1) {{\footnotesize $i'$}};	
	
	\draw[] (2) to [bend left=50] (4.5,0.4);
	\draw[] (0.5,0.4) to [bend left=20] (1);
   \draw[] (4) to [bend left=30] (4.5,0.9);
	\draw[] (0.5,0.9) to [bend left=40] (3);   
   
   \node[] (w1) at (0.5,-0.6) {};
   \node[] (w2) at (0.5,1.8) {};
   \node[] (w3) at (4.5,-0.6) {};
   \node[] (w4) at (4.5,1.8) {};
	\draw[dashed] (w1) -- (w2);  
	\draw[dashed] (w3) -- (w4);  	

	\node[] (l) at (-0.7,0.5) {(6)};
\end{scope}
\end{tikzpicture}
\end{center}
    \caption{The six configurations of cyclically crossing arcs}
    \label{fig:cycCrossing}
\end{figure}

A pair of arcs that do not cyclically cross are said to be \textbf{cyclically non-crossing}.
Note that a cyclic permutation of the indices preserves cyclically crossing arcs, and thereby it also preserves cyclically non-crossing arcs.

A \textbf{cyclic $(I,J)$-forest} is a subgraph of the complete bipartite graph $K_{I,J}$ whose arcs are cyclically non-crossing.
A \textbf{cyclic $(I,J)$-tree} is a maximal cyclic $(I,J)$-forest.
It will be useful to consider the subset of cyclic $(I,J)$-trees with only increasing arcs. 
We therefore refer to such cyclic $(I,J)$-trees as \textbf{increasing $(I,J)$-trees}. 
The cyclic $(I,J)$-tree on the left in Figure~\ref{fig:IJtrees} is also an increasing $(I,J)$-tree.

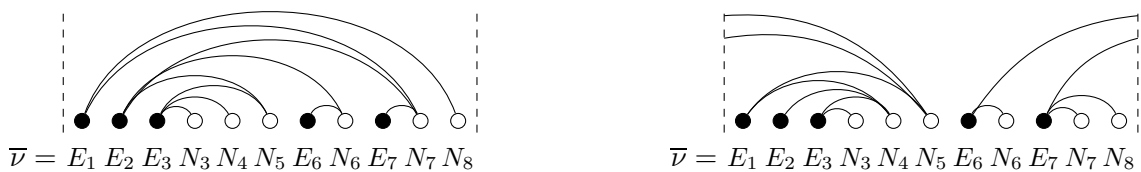
\begin{figure}
\begin{center}
\begin{tikzpicture}
\begin{scope}[xshift = 0, yshift = 0, scale=0.5]
	\node[style={circle,draw, inner sep=2pt, fill=black}] (1)  at (1,0)  {};
	\node[style={circle,draw, inner sep=2pt, fill=black}] (2)  at (2,0)  {};
	\node[style={circle,draw, inner sep=2pt, fill=black}] (3)  at (3,0)  {};
	\node[style={circle,draw, inner sep=2pt, fill=none}] (4)  at (4,0)  {};
   \node[style={circle,draw, inner sep=2pt, fill=none}] (5)  at (5,0)  {};	
	\node[style={circle,draw, inner sep=2pt, fill=none}] (6)  at (6,0)  {};
	\node[style={circle,draw, inner sep=2pt, fill=black}] (7)  at (7,0)  {};
	\node[style={circle,draw, inner sep=2pt, fill=none}] (8)  at (8,0)  {};
	\node[style={circle,draw, inner sep=2pt, fill=black}] (9)  at (9,0)  {};
	\node[style={circle,draw, inner sep=2pt, fill=none}] (10)  at (10,0)  {};	
	\node[style={circle,draw, inner sep=2pt, fill=none}] (11)  at (11,0)  {};		
	
	\node[] (nu)  at (-0.3,-0.95)  {\textcolor{black}{{\small $\overline{\nu}=$}}};	
		\node[] (E1)  at (1,-1)  {\textcolor{black}{{\footnotesize $E_1$}}};
	\node[] (E2)  at (2,-1)  {\textcolor{black}{{\footnotesize $E_2$}}};
	\node[] (E3)  at (3,-1)  {\textcolor{black}{{\footnotesize $E_3$}}};
    \node[] (N1)  at (4,-1)  {\textcolor{black}{{\footnotesize $N_3$}}};	
	\node[] (N2)  at (5,-1)  {\textcolor{black}{{\footnotesize $N_4$}}};
	\node[] (N3)  at (6,-1)  {\textcolor{black}{{\footnotesize $N_5$}}};	
	\node[] (E4)  at (7,-1)  {\textcolor{black}{{\footnotesize $E_6$}}};	
	\node[] (E5)  at (8,-1)  {\textcolor{black}{{\footnotesize $N_6$}}};
	\node[] (N4)  at (9,-1)  {\textcolor{black}{{\footnotesize $E_7$}}};	
	\node[] (E5)  at (10,-1)  {\textcolor{black}{{\footnotesize $N_{7}$}}};
	\node[] (N5)  at (11,-1)  {\textcolor{black}{{\footnotesize $N_{8}$}}};

   \draw[] (1) to [bend left=70] (11);
   \draw[] (1) to [bend left=65] (10);
   \draw[] (2) to [bend left=60] (10);
   \draw[] (2) to [bend left=65] (6);
   \draw[] (2) to [bend left=65] (8);
   \draw[] (3) to [bend left=60] (4);
   \draw[] (3) to [bend left=60] (5);
   \draw[] (3) to [bend left=65] (6);
   \draw[] (7) to [bend left=65] (8);
   \draw[] (9) to [bend left=65] (10);
   
   \node[] (w1) at (0.5,-0.6) {};
   \node[] (w2) at (0.5,3.3) {};
   \node[] (w3) at (11.5,-0.6) {};
   \node[] (w4) at (11.5,3.3) {};
	\draw[dashed] (w1) -- (w2);  
	\draw[dashed] (w3) -- (w4);  	   
   
\end{scope}

\begin{scope}[xshift = 250, yshift = 0, scale=0.5]
	\node[style={circle,draw, inner sep=2pt, fill=black}] (1)  at (1,0)  {};
	\node[style={circle,draw, inner sep=2pt, fill=black}] (2)  at (2,0)  {};
	\node[style={circle,draw, inner sep=2pt, fill=black}] (3)  at (3,0)  {};
	\node[style={circle,draw, inner sep=2pt, fill=none}] (4)  at (4,0)  {};
   \node[style={circle,draw, inner sep=2pt, fill=none}] (5)  at (5,0)  {};	
	\node[style={circle,draw, inner sep=2pt, fill=none}] (6)  at (6,0)  {};
	\node[style={circle,draw, inner sep=2pt, fill=black}] (7)  at (7,0)  {};
	\node[style={circle,draw, inner sep=2pt, fill=none}] (8)  at (8,0)  {};
	\node[style={circle,draw, inner sep=2pt, fill=black}] (9)  at (9,0)  {};
	\node[style={circle,draw, inner sep=2pt, fill=none}] (10)  at (10,0)  {};	
	\node[style={circle,draw, inner sep=2pt, fill=none}] (11)  at (11,0)  {};		
	
	\node[] (nu)  at (-0.3,-0.95)  {\textcolor{black}{{\small $\overline{\nu}=$}}};	
	\node[] (E1)  at (1,-1)  {\textcolor{black}{{\footnotesize $E_1$}}};
	\node[] (E2)  at (2,-1)  {\textcolor{black}{{\footnotesize $E_2$}}};
	\node[] (E3)  at (3,-1)  {\textcolor{black}{{\footnotesize $E_3$}}};
    \node[] (N1)  at (4,-1)  {\textcolor{black}{{\footnotesize $N_3$}}};	
	\node[] (N2)  at (5,-1)  {\textcolor{black}{{\footnotesize $N_4$}}};
	\node[] (N3)  at (6,-1)  {\textcolor{black}{{\footnotesize $N_5$}}};	
	\node[] (E4)  at (7,-1)  {\textcolor{black}{{\footnotesize $E_6$}}};	
	\node[] (E5)  at (8,-1)  {\textcolor{black}{{\footnotesize $N_6$}}};
	\node[] (N4)  at (9,-1)  {\textcolor{black}{{\footnotesize $E_7$}}};	
	\node[] (E5)  at (10,-1)  {\textcolor{black}{{\footnotesize $N_{7}$}}};
	\node[] (N5)  at (11,-1)  {\textcolor{black}{{\footnotesize $N_{8}$}}};

   \draw[] (1) to [bend left=55] (6);
   \draw[] (1) to [bend left=55] (5);
   \draw[] (2) to [bend left=55] (5);      
   \draw[] (3) to [bend left=50] (4);
   \draw[] (3) to [bend left=60] (5);
   \draw[] (7) to [bend left=65] (8);
   \draw[] (7) to [bend left=25] (11.5,2.8);
   \draw[] (0.5,2.8) to [bend left=30] (6);
   \draw[] (9) to [bend left=25] (11.5,2.2);
   \draw[] (0.5,2.2) to [bend left=30] (6);

   \draw[] (9) to [bend left=55] (10);
   \draw[] (9) to [bend left=65] (11);   
   
   \node[] (w1) at (0.5,-0.6) {};
   \node[] (w2) at (0.5,3.3) {};
   \node[] (w3) at (11.5,-0.6) {};
   \node[] (w4) at (11.5,3.3) {};
	\draw[dashed] (w1) -- (w2);  
	\draw[dashed] (w3) -- (w4);  	   
   
\end{scope}
\end{tikzpicture}
\end{center}
    \caption{Two cyclic $(I,J)$-trees, with $I$ and $J$ determined by $\overline{\nu}$. The left tree has maximal arc $(1,8)$, while the right tree has maximal arc $(6,5)$}
    \label{fig:IJtrees}
\end{figure}

For an increasing $(I,J)$-tree $T$, let $\mu_T$ denote the unique lattice path whose number of $E$ steps at height $k$ is the indegree of the $k$th element of $J$ (in the order determined by $\overline{\nu}$).

\begin{lemma}(\cite[Proposition 3.3]{CPS19})
\label{lem:bIJectionCat(nu)}
The map given by $T\mapsto \mu_T$ is a bijection between the increasing $(I,J)$-trees determined by $\nu$ and the lattice paths from $(0,0)$ to $(a,b)$ staying weakly above $\nu$. 
In particular, the number of increasing $(I,J)$-trees is $\Cat(\nu)$. \qed
\end{lemma}

\begin{definition}
The \textbf{cyclic $\nu$-Tamari complex} is the flag simplicial complex of cyclic $(I,J)$-forests determined by $\nu$ whose minimal non-faces are pairs of cyclically crossing arcs. 
\end{definition}

The facets of the cyclic $\nu$-Tamari complex are its cyclic $(I,J)$-trees. 
Apart from the maximal arc, each arc in a cyclic $(I,J)$-tree is the longest arc incident to either of its end points.
Hence the number of arcs in a cyclic $(I,J)$-tree is $\abs{I}+\abs{J}+1$. 
As an immediate consequence we have the following.

\begin{lemma}
\label{lem:cycTamDim}
The cyclic $\nu$-Tamari complex is pure with dimension $a+b = \abs{I} + \abs{J} - 2$. \qed
\end{lemma}

Two cyclic $(I,J)$-trees $T$ and $T'$ are related by an \textbf{increasing flip} if $T'$ can be obtain from $T$ by replacing an arc $(i,j)\in T$ with an arc $(i',j')$ where $i<i'$. 
This gives a cover relation $T<_{I,J}T'$ on cyclic $(I,J)$-trees.  
The \textbf{cyclic $\nu$-Tamari poset} is the transitive closure of the relation $<_{I,J}$ on the set of cyclic $(I,J)$-trees \cite[Lemma 7.1]{CPS19}.	
Restricting to the set of increasing $(I,J)$-trees, the transitive closure of $<_{I,J}$ gives the \textbf{$\nu$-Tamari lattice} \cite[Proposition 3.5]{CPS19}.

A directed graph is said to be \textbf{alternating} if each of its vertices is a source or sink. 
Let $D_{\overline{\nu}}$ denote the set of all maximal alternating graphs on vertex set $I\cup J$ using forward edges $(i,j)$ satisfying $i\in I$ and $j \in J$.
Note that by maximality the graphs in $D_{\overline{\nu}}$ must be trees and therefore have $\abs{I\cup J} -1$ edges.
Recall that applying the permutation $\sigma_i$ from Lemma~\ref{lem:sigma} to the indices of $\overline{\nu}(i)$ gives its canonical labeling.
Thus by a slight abuse of notation, we let $\sigma_i \overline{\nu}(i)$ denote the canonically labeled $\overline{\nu}(i)$ with its $E$ steps indexed by $\sigma_iI := \{\sigma_i(k)\mid k\in I\}$ and $N$ steps indexed by $\sigma_iJ:=\{\sigma_i(k)\mid k\in J\}$.
In addition, let $D_{\overline{\nu}(i)}^*$ denote the set of graphs in $D_{\sigma_i\overline{\nu}(i)}$, but where the permutation $\sigma_i^{-1}$ has been applied to the vertices in each graph\footnote{In the case when $\nu = (NE)^n$ the graphs in $\cup_{i\in [w]} D^*_{\overline{\nu}(i)}$ are the \emph{valid digraphs} of Ehrenborg, Hetyei, and Readdy \cite{EHR18}.}.

\begin{lemma}
\label{lem:bijectionDnutoIJ}
There is a bijection $\Phi$ between the graphs in $\cup_{i\in [w]} D_{\overline{\nu}(i)}^*$ and the set of cyclic $(I,J)$-trees determined by $\overline{\nu}$.
In addition, two graphs $G_1$ and $G_2$ in $\cup_{i\in [w]} D_{\overline{\nu}(i)}^*$ differ by a single edge if and only if $\Phi(G_1)$ and $\Phi(G_2)$ differ by a single arc.  
\end{lemma}

\begin{proof}
Let $G$ be a graph in $\cup_{i\in [w]} D_{\overline{\nu}(i)}^*$, and let $T_G$ denote the subgraph of $K_{I,J}$ induced by the arcs $\{(k,\ell) \mid (k,\ell) \in E(G)\}\cup \{(k,k)\mid k\in I\cap J\}$.
We define $\Phi$ by $G\mapsto T_G$, and verify that $T_G$ is in fact a cyclic $(I,J)$-tree.
By the construction of $G$ there is a unique $i \in [w]$ for which the edges in $\sigma_iG$ are forward edges, with $\sigma_iG$ denoting the graph obtained from $G$ by permuting its vertices by $\sigma_i$. 
Note that the longest edge in $G$ is pointed toward this $i$. 
For example, in Figure~\ref{fig:PhiEx} the graph $G$ has longest edge $(3,2)$, so $\sigma_2G$ has only forward edges.
The arcs $\{(\sigma_i(k),\sigma_i(\ell)) \mid (k,\ell) \in E(G)\} \cup \{(\sigma_i(k),\sigma_i(k)) \mid k \in I\cap J\}$ are now non-crossing and form an $(\sigma_iI,\sigma_iJ)$-forest $T_{\sigma_iG}$. 
Moreover, the number of arcs in $T_{\sigma_iG}$ is $\abs{\sigma_iI \cup \sigma_iJ} -1 + \abs{\sigma_iI\cap \sigma_iJ} = \abs{\sigma_iI} + \abs{\sigma_iJ} - 1$, and so $T_{\sigma_iG}$ is a cyclic $(\sigma_iI,\sigma_iJ)$-tree. 
Since a cyclic permutation of indices by $\sigma_i\inv$ preserves cyclic non-crossing condition on arcs, $T_G$ is a cyclic $(I,J)$-tree. 

Define $\Phi\inv$ by mapping the cyclic $(I,J)$-tree $T$ to the graph $G_T$ induced by the edge set $\{(k,\ell)\mid (k,\ell) \text{ is a non-minimal arc in }T \}$. 
To verify that $\Phi\inv$ is well-defined, we show that $G_T \in \cup_{i\in [w]} D_{\overline{\nu}(i)}^*$. 
The maximal arc in $T$ begins at the vertex corresponding the $E$ step of the $i$th cyclic peak for some $i\in [w]$. 
Let $\sigma_iT$ denote the $(\sigma_iI,\sigma_iJ)$-tree with arcs $\{(\sigma_i(k),\sigma_i(\ell)) \mid (k,\ell) \text{ is an arc in }T \}$.
Then the arcs of $\sigma_iT$ are increasing, with its vertices given by the path $\sigma_i\overline{\nu}(i)$. 
Let $\sigma_iG_T$ be the graph induced by the edge set $\{(\sigma_i(k),\sigma_i(\ell)) \mid (k,\ell) \text{ is a non-minimal arc in } T \}$. 
The edges of $\sigma_iG_T$ are necessarily increasing, non-crossing, and alternating. 
Furthermore, the number of edges in $\sigma_iG_T$ is $\abs{\sigma_iI\cup \sigma_iJ} -1$, so it is a maximal alternating graph.
It follows that $G_T \in D^*_{\overline{\nu}(i)}$.

Since $\Phi\circ \Phi\inv$ and $\Phi\inv \circ \Phi$ are identity maps, $\Phi$ is a bijection. 
It is immediate by the construction of $\Phi$ that two graphs in the domain of $\Phi$ differ by a single edge if and only if their images under $\Phi$ differ by a single arc. 
\end{proof}

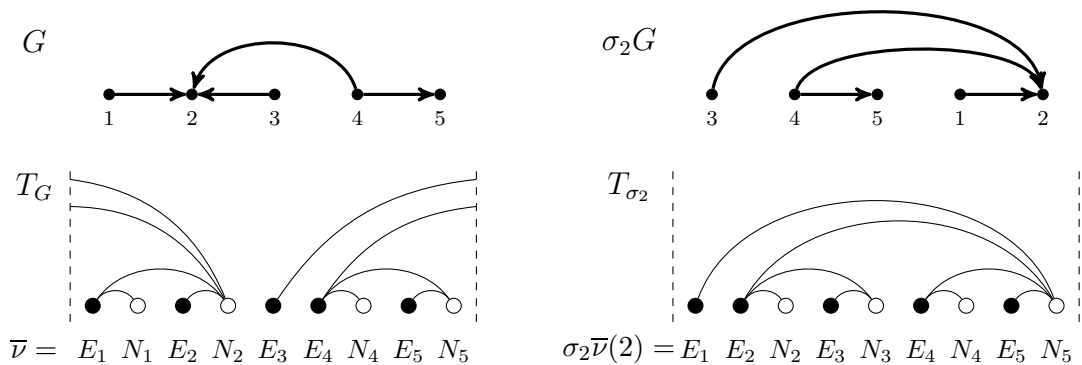
\begin{figure}
\begin{center}
\begin{tikzpicture}
\begin{scope}[xshift=-28, yshift=0, scale=1.1]
	\vertex[fill,label=below:\tiny{$1$}](a1) at (1,0) {};
	\vertex[fill,label=below:\tiny{$2$}](a2) at (2,0) {};
	\vertex[fill,label=below:\tiny{$3$}](a3) at (3,0) {};
	\vertex[fill,label=below:\tiny{$4$}](a4) at (4,0) {};
	\vertex[fill,label=below:\tiny{$5$}](a5) at (5,0) {};

    \draw[>=stealth', ->, very thick] (a1)--(a2);
	\draw[>=stealth', ->, very thick] (a3)--(a2);
	\draw[>=stealth', <-, very thick] (a2) .. controls (2.25, 0.8) and (3.75, 0.8) .. (a4);
	\draw[>=stealth', ->, very thick] (a4)--(a5);
  
\end{scope}

\begin{scope}[xshift = -20, yshift = -80, scale=0.6]
	\node[style={circle,draw, inner sep=2pt, fill=black}] (1)  at (1,0)  {};
	\node[style={circle,draw, inner sep=2pt, fill=none}] (2)  at (2,0)  {};
	\node[style={circle,draw, inner sep=2pt, fill=black}] (3)  at (3,0)  {};
	\node[style={circle,draw, inner sep=2pt, fill=none}] (4)  at (4,0)  {};
   \node[style={circle,draw, inner sep=2pt, fill=black}] (5)  at (5,0)  {};	
	\node[style={circle,draw, inner sep=2pt, fill=black}] (6)  at (6,0)  {};
	\node[style={circle,draw, inner sep=2pt, fill=none}] (7)  at (7,0)  {};
	\node[style={circle,draw, inner sep=2pt, fill=black}] (8)  at (8,0)  {};
	\node[style={circle,draw, inner sep=2pt, fill=none}] (9)  at (9,0)  {};
	
	\node[] (nu)  at (-0.3,-0.95)  {\textcolor{black}{{\small $\overline{\nu}=$}}};	
	\node[] (E1)  at (1,-1)  {\textcolor{black}{{\footnotesize $E_1$}}};
	\node[] (E2)  at (2,-1)  {\textcolor{black}{{\footnotesize $N_1$}}};
	\node[] (E3)  at (3,-1)  {\textcolor{black}{{\footnotesize $E_2$}}};
    \node[] (N1)  at (4,-1)  {\textcolor{black}{{\footnotesize $N_2$}}};
	\node[] (N2)  at (5,-1)  {\textcolor{black}{{\footnotesize $E_3$}}};
	\node[] (N3)  at (6,-1)  {\textcolor{black}{{\footnotesize $E_4$}}};
	\node[] (E4)  at (7,-1)  {\textcolor{black}{{\footnotesize $N_4$}}};
	\node[] (E5)  at (8,-1)  {\textcolor{black}{{\footnotesize $E_5$}}};
	\node[] (N4)  at (9,-1)  {\textcolor{black}{{\footnotesize $N_5$}}};	
	
    \draw[] (1) to [bend left=55] (2);   
    \draw[] (1) to [bend left=55] (4);       
    \draw[] (3) to [bend left=55] (4);     
    \draw[] (6) to [bend left=55] (7);        
    \draw[] (8) to [bend left=55] (9);
    \draw[] (6) to [bend left=55] (9);

   \draw[] (5) to [bend left=25] (9.5,2.8);
   \draw[] (0.5,2.8) to [bend left=30] (4);
   \draw[] (6) to [bend left=25] (9.5,2.2);
   \draw[] (0.5,2.2) to [bend left=30] (4);

   \node[] (w1) at (0.5,-0.6) {};
   \node[] (w2) at (0.5,3.3) {};
   \node[] (w3) at (9.5,-0.6) {};
   \node[] (w4) at (9.5,3.3) {};
	\draw[dashed] (w1) -- (w2);  
	\draw[dashed] (w3) -- (w4);  	   
   
\end{scope}

\begin{scope}[xshift=200, yshift=0, scale=1.1]
	\vertex[fill,label=below:\tiny{$3$}](a1) at (1,0) {};
	\vertex[fill,label=below:\tiny{$4$}](a2) at (2,0) {};
	\vertex[fill,label=below:\tiny{$5$}](a3) at (3,0) {};
	\vertex[fill,label=below:\tiny{$1$}](a4) at (4,0) {};
	\vertex[fill,label=below:\tiny{$2$}](a5) at (5,0) {};

	\draw[>=stealth', ->, very thick] (a2)--(a3);
	\draw[>=stealth', ->, very thick] (a2) .. controls (2.25, 0.7) and (4.75, 0.7) .. (a5);
	\draw[>=stealth', ->, very thick] (a1) .. controls (1.25, 1.3) and (4.75, 1.3) .. (a5);
	\draw[>=stealth', ->, very thick] (a4)--(a5);
	
\end{scope}

\begin{scope}[xshift = 208, yshift = -80, scale=0.6]
	\node[style={circle,draw, inner sep=2pt, fill=black}] (1)  at (1,0)  {};
	\node[style={circle,draw, inner sep=2pt, fill=black}] (2)  at (2,0)  {};
	\node[style={circle,draw, inner sep=2pt, fill=none}] (3)  at (3,0)  {};
	\node[style={circle,draw, inner sep=2pt, fill=black}] (4)  at (4,0)  {};
   \node[style={circle,draw, inner sep=2pt, fill=none}] (5)  at (5,0)  {};	
	\node[style={circle,draw, inner sep=2pt, fill=black}] (6)  at (6,0)  {};
	\node[style={circle,draw, inner sep=2pt, fill=none}] (7)  at (7,0)  {};
	\node[style={circle,draw, inner sep=2pt, fill=black}] (8)  at (8,0)  {};
	\node[style={circle,draw, inner sep=2pt, fill=none}] (9)  at (9,0)  {};
	
	\node[] (nu)  at (-0.7,-0.95)  {\textcolor{black}{{\small $\sigma_2\overline{\nu}(2)=$}}};	
	\node[] (E1)  at (1,-1)  {\textcolor{black}{{\footnotesize $E_1$}}};
	\node[] (E2)  at (2,-1)  {\textcolor{black}{{\footnotesize $E_2$}}};
	\node[] (E3)  at (3,-1)  {\textcolor{black}{{\footnotesize $N_2$}}};
    \node[] (N1)  at (4,-1)  {\textcolor{black}{{\footnotesize $E_3$}}};
	\node[] (N2)  at (5,-1)  {\textcolor{black}{{\footnotesize $N_3$}}};
	\node[] (N3)  at (6,-1)  {\textcolor{black}{{\footnotesize $E_4$}}};
	\node[] (E4)  at (7,-1)  {\textcolor{black}{{\footnotesize $N_4$}}};
	\node[] (E5)  at (8,-1)  {\textcolor{black}{{\footnotesize $E_5$}}};
	\node[] (N4)  at (9,-1)  {\textcolor{black}{{\footnotesize $N_5$}}};	
	
    \draw[] (2) to [bend left=55] (3);   
    \draw[] (4) to [bend left=55] (5);
    \draw[] (6) to [bend left=55] (7);    
    \draw[] (8) to [bend left=55] (9);
    \draw[] (1) to [bend left=70] (9);
    \draw[] (2) to [bend left=60] (9);
    \draw[] (2) to [bend left=55] (5);
    \draw[] (6) to [bend left=55] (9);

   \node[] (w1) at (0.5,-0.6) {};
   \node[] (w2) at (0.5,3.3) {};
   \node[] (w3) at (9.5,-0.6) {};
   \node[] (w4) at (9.5,3.3) {};
	\draw[dashed] (w1) -- (w2);  
	\draw[dashed] (w3) -- (w4);  	   
   
\end{scope}

\begin{scope}[xshift=-25, yshift=20, scale=1]
\node[] (a) at (0,0) {$G$};
\end{scope}

\begin{scope}[xshift=-25, yshift=-35, scale=1]
\node[] (a) at (0,0) {$T_G$};
\end{scope}

\begin{scope}[xshift=200, yshift=20, scale=1]
\node[] (a) at (0,0) {$\sigma_2G$};
\end{scope}

\begin{scope}[xshift=200, yshift=-35, scale=1]
\node[] (a) at (0,0) {$T_{\sigma_2}$};
\end{scope}

\end{tikzpicture}
\end{center}
    \caption{A graph $G$ and the cyclic $(I,J)$-tree $\Phi(G) = T_G$. 
Applying $\sigma_2$ to the vertices of $G$ gives the graph $\sigma_2G$ with only has forward edges. 
The cyclic $(\sigma_2I,\sigma_2J)$-tree $T_{\sigma_2G}$ then has only increasing arcs}
    \label{fig:PhiEx}
\end{figure}

We will now construct a reduction order such that the reduced form of the polynomial $P_\nu$ encodes the cyclic $\nu$-Tamari complex. 
Define the length of an edge $(i,j)$ in a digraph to be $j-i \pmod n$. 
A pair of edges $\{(i,j),(j,k)\}$ (identified with $x_{ij}x_{jk}$) form a \textbf{longest pair at vertex $j$} if $x_{ij}$ and $x_{jk}$ respectively correspond with the longest incoming and outgoing edges at $j$.
The \textbf{length reduction order} $\rho_{\text{len}}$ is the reduction order induced by successively reducing the longest pair $x_{ij}x_{jk}$ with minimal $j$ at each reduction step.
To see that reducing $P_\nu$ according to $\rho_{\text{len}}$ encodes the cyclic $\nu$-Tamari complex, we first need a few lemmas. 

\begin{lemma}
For a non-crossing graph $G$ with only forward edges, the graphs produced by a reduction at a longest pair are also non-crossing.
\label{lem:non-crossing}
\end{lemma}

\begin{proof}
Let $G$ be a non-crossing graph on $[n]$ with all edges directed forward. 
Consider any vertex $j$ with its longest incoming and outgoing edges $(i,j)$ and $(j,k)$ where $i<j<k$, i.e. there are no edges $(i',j)$ and $(j,k')$ with $i'<i$ and $k<k'$.   
Since no edge in $G$ crosses $(i,j)$ or $(j,k)$, the edge $(i,k)$ does not cross any edge in $G$. 
Thus each of the three graphs produced by a reduction at $j$ are non-crossing, and so a reduction produces only non-crossing graphs. 
\end{proof}

\begin{lemma}
\label{lem:anyorder}
All reduced forms of $M(G(\nu(i)))$ obtained with a reduction order in which longest pairs are reduced first at each vertex are equal. 
Furthermore, if $R$ is the unique such reduced form, then the highest degree terms in $R$ encode the graphs in $D^*_{\overline{\nu}(i)}$.
\end{lemma}

\begin{proof}
Consider the reduction tree $\scrT$ whose leaves encode $R$, and let $\rho$ denote the reduction order used.
Note that the length of an edge is preserved under a cyclic relabeling of the vertices of $G(\nu(i))$. Therefore, permuting the vertices of the graphs in $\scrT$ by $\sigma_i$ is a reduction tree $\sigma_i\scrT$ for $G(\sigma_i\nu(i))$ in which longest pairs are reduced first at each vertex.
Let $\sigma_iR$ denote the reduced form corresponding to the leaves of $\sigma_i\scrT$. 
Since the edges of $G(\sigma_i\nu(i))$ are forward edges and non-crossing, by Lemma~\ref{lem:non-crossing}, the leaves of $\sigma_i\scrT$ are non-crossing graphs.  
In addition, the leaves of $\sigma_i\scrT$ do not have multiedges since $G(\sigma_i\nu(i))$ is acyclic as an undirected graph. 
The number of edges in $G(\sigma_i\nu(i))$ is $\abs{\sigma_iI\cup \sigma_iJ}-1$. 
Since simple reductions preserve the number of edges,
the monomials of highest degree in $\sigma_iR$ encode simple alternating graphs on vertex set $\sigma_iI\cup \sigma_iJ$ with $\abs{\sigma_iI\cup \sigma_iJ}-1$ edges, which must therefore be maximal. 

Since the monomials of highest degree in $\sigma_iR$ correspond with the facets of a unimodular triangulation of $\calF_{\widehat{G}(\sigma_i\nu(i))}$, they are enumerated by the volume of $\calF_{\widehat{G}(\sigma_i\nu(i))}$, which by Lemma~\ref{lem:nuCarVol} is $\Cat(\sigma_i\nu(i))$. 
By Lemma~\ref{lem:bIJectionCat(nu)}, the number of cyclic $(\sigma_iI,\sigma_iJ)$-trees using only increasing arcs is also counted by $\Cat(\sigma_i\nu(i))$. 
Since the bijection $\Phi$ of Lemma~\ref{lem:bijectionDnutoIJ} restricts to a bijection between $D_{\sigma_i\overline{\nu}(i)}$ and the cyclic $(\sigma_iI,\sigma_iJ)$-trees with increasing arcs, the monomials of highest degree in $\sigma_iR$ must comprise the whole set $D_{\sigma_i\overline{\nu}(i)}$.
Since a reduced form is determined by its highest degree terms, $\sigma_iR$ is determined by the graphs in $D_{\sigma_i\overline{\nu}(i)}$.

Applying $\sigma^{-1}$ to all graphs in $\sigma_i\scrT$, we recover $\scrT$ and the reduced form $R$. 
The graphs encoded by the the highest degree terms of $R$ are the graphs obtained by applying $\sigma_i^{-1}$ to the graphs in $D_{\sigma_i\overline{\nu}(i)}$, which are precisely the graphs in $D^*_{\overline{\nu}(i)}$.

\end{proof}

\begin{theorem}
\label{thm:lengthReduction}
The triangulation of $\calF_{\widehat{G}_B(\nu)}\equiv \Delta_a\times \Delta_b$ encoded by the reduced form of $P_\nu$ in the reduction order $\rho_{\text{len}}$ is a geometric realization of the cyclic $\nu$-Tamari complex.
\end{theorem}

\begin{proof}
Let $R$ be the reduced form obtained from reducing $P_\nu$ in the order $\rho_{\text{len}}$. 
For convenience we assume that simple reductions were done at each step, so $R$ has only highest degree terms.
While reducing $P_\nu$ in the order $\rho_{\text{len}}$, each $M(G(\nu(i)))$ is reduced in some reduction order that reduces longest pairs first. 
Therefore, by Lemma~\ref{lem:anyorder}, $R$ is equivalently obtained by separately reducing each $M(G(\nu(i)))$ using $\rho_{\text{len}}$, and so the terms of $R$ record the graphs in $\cup_{i\in [w]} D^*_{\overline{\nu}(i)}$.
It then follows from Lemma~\ref{lem:bijectionDnutoIJ} that two monomials in $R$ record adjacent facets (facets sharing a codimension one face) in the cyclic $\nu$-Tamari complex if and only if they differ by a single generator.
Two facets in the triangulation $\calT$ of $\calF_{\widehat{G}_B(\nu)}$ encoded by $R$ are also adjacent if and only if their corresponding monomials in $R$ differ by a single generator. 
It follows that $\calT$ (as a simplicial complex) and the cyclic $\nu$-Tamari complex have the same dual graph. 
Since both are flag simplicial complexes of equal dimension, they are necessarily isomorphic. 
\end{proof}

\begin{corollary}
Reducing $P_\nu$ in the order $\rho_{\text{len}}$ gives a triangulation of $\calF_{\widehat{G}_B(\nu)}$ whose dual graph is the Hasse diagram of the cyclic $\nu$-Tamari poset.
\end{corollary}

The reduction order in Example~\ref{ex.a3b2} was $\rho_{len}$ and hence the dual graph of the resulting triangulation in Figure~\ref{fig:dualGraph} is the Hasse diagram of the cyclic $\nu$-Tamari poset for $\nu = NEENE$.

\begin{figure}
\begin{center}
\begin{tikzpicture}[scale=0.8]
\begin{scope}[scale=1.5, xshift=0, yshift=0]
\tikzstyle{vertex}=[circle, fill=black, inner sep=0pt, minimum size=3pt]

	\vertex (a0) at (0,0) {};
	\vertex (a1) at (-1,-0.5) {};
	\vertex (a2) at (1,-0.5) {};
	\vertex (a3) at (0,1.3) {};
	
	\draw[thick,fill=pur,opacity=0.6] (-1,-0.5) -- (1,-0.5) -- (0,1.3) -- (-1,-0.5);	
		
	\draw[thick,dashed] (a0) -- (a1);
	\draw[thick,dashed] (a0) -- (a3);
	\draw[thick,dashed] (a0) -- (a2);	
	\draw[thick] (a1) -- (a2) -- (a3) -- (a1);

	\vertex (a1) at (-1,-0.5) {};
	\vertex (a2) at (1,-0.5) {};
	\vertex (a3) at (0,1.3) {};

	\node[] (x1) at (1,-0.7) {{\small $\mathbf{x_{12}}x_{34}\mathbf{x_{41}}$}};
	\node[] (x2) at (-0,1.5) {{\small $\mathbf{x_{12}}x_{23}\mathbf{x_{41}}$}};
	\node[] (x3) at (-1,-0.7) {{\small $x_{23}x_{34}x_{41}$}};
	\node[] (x4) at (-0.7,0.15) {{\small $x_{12}x_{23}x_{34}$}};
\end{scope}

\begin{scope}[scale=1.5, xshift=45, yshift=11]
	\node[] (k) at (0,0) {$\longrightarrow$};
\end{scope}

\begin{scope}[scale=1.5, xshift=145, yshift=11]
	\node[] (k) at (0,0) {$\longrightarrow$};
\end{scope}

\begin{scope}[scale=1.5, xshift=170, yshift=11]
	\node[] (k) at (0,0) {$\cdots$};
\end{scope}

\begin{scope}[scale=1.5, xshift=195, yshift=11]
	\node[] (k) at (0,0) {$\longrightarrow$};
\end{scope}

\begin{scope}[scale=1.5, xshift=100, yshift=0]
\tikzstyle{vertex}=[circle, fill=black, inner sep=0pt, minimum size=3pt]

	\vertex (a0) at (0,0) {};
	\vertex (a1) at (-1,-0.5) {};
	\vertex (a21) at (0.6,-0.5) {};
	\vertex (a22) at (0.8,-0.3) {};	
	\vertex (a31) at (0,1.1) {};
	\vertex (a32) at (-0.2,0.9) {};

	\draw[thick,dashed] (a0) -- (a1);
	\draw[thick,dashed] (a0) -- (a31);
	\draw[thick,dashed] (a0) -- (a22);	

	\draw[fill=pur,opacity=0.6,thick] (-1,-0.5) -- (0.6,-0.5) -- (0.8,-0.3) -- (0,1.1) -- (-0.2,0.9) -- (-1,-0.5);	

	\draw[thick] (a1) -- (a21);
	\draw[thick] (a1) -- (a32);
	\draw[thick] (a31) -- (a22);	
	\draw[thick] (a21) -- (a22);
	\draw[thick] (a31) -- (a32);	
	\draw[thick] (a32) -- (a21);

	\node[] (x1) at (1,-0.7) {{\small $x_{34}x_{41}x_{42}$}};
	\node[] (x1) at (1.55,-0.3) {{\small $x_{12}x_{34}x_{42}$}};	
	\node[] (x2) at (-0.1,1.3) {{\small $\mathbf{x_{12}x_{23}}x_{42}$}};
	\node[] (x2) at (-0.95,1.0) {{\small $x_{23}x_{41}x_{42}$}};	
	\node[] (x3) at (-1,-0.7) {{\small $x_{23}x_{34}x_{41}$}};
	\node[] (x4) at (-0.73,0.15) {{\small $\mathbf{x_{12}x_{23}}x_{34}$}};
\end{scope}

\begin{scope}[scale=1.9, xshift=200, yshift=0]
\tikzstyle{vertex}=[circle, fill=black, inner sep=0pt, minimum size=3pt]

	\vertex (a01) at (-0.2,-0.2) {};
	\vertex (a02) at (-0.3,0) {};	
	\vertex (a03) at (-0.1,0.2) {};
	\vertex (a04) at (0.15,0.2) {};
	\vertex (a05) at (0.3,-0.2) {};	

	\vertex (a11) at (-0.7,-0.4) {};
	\vertex (a12) at (-0.8,-0.2) {};
	\vertex (a13) at (-0.7,0.1) {};
	\vertex (a14) at (-0.4,-0.35) {};
	\vertex (a15) at (-0.5,-0.5) {};		

	\vertex (a21) at (0.7,-0.4) {};
	\vertex (a22) at (0.8,-0.0) {};
	\vertex (a23) at (0.6,-0.1) {};
	\vertex (a24) at (0.4,-0.35) {};
	\vertex (a25) at (0.5,-0.5) {};		
	
	\vertex (a31) at (0.15,0.85) {};
	\vertex (a32) at (-0.12,0.85) {};	
	\vertex (a33) at (0.45,0.8) {};
	\vertex (a34) at (0.25,0.7) {};
	\vertex (a35) at (-0.35,0.7) {};		
	\draw[thick, dashed] (a01) -- (a02) -- (a03) -- (a04) -- (a05) -- (a01);
	
	\draw[thick, dashed] (a31) -- (a04);
	\draw[thick, dashed] (a32) -- (a03);

	\draw[thick, dashed] (a21) -- (a05);

	\draw[thick, dashed] (a12) -- (a02);
	\draw[thick, dashed] (a11) -- (a01);

	\draw[thick,fill=pur,opacity=0.6] (-0.5,-0.5) -- (0.5,-0.5) -- (0.7,-0.4) -- (0.8,-0.0) -- (0.45,0.8) -- (0.18,0.85) -- (-0.12,0.85) -- (-0.35,0.7) -- (-0.7,0.1) -- (-0.8,-0.2) -- (-0.7,-0.4) -- (-0.5,-0.5);

	\vertex (a11) at (-0.7,-0.4) {};
	\vertex (a12) at (-0.8,-0.2) {};
	\vertex (a13) at (-0.7,0.1) {};
	\vertex (a14) at (-0.4,-0.35) {};
	\vertex (a15) at (-0.5,-0.5) {};		

	\vertex (a21) at (0.7,-0.4) {};
	\vertex (a22) at (0.8,-0.0) {};
	\vertex (a23) at (0.6,-0.1) {};
	\vertex (a24) at (0.4,-0.35) {};
	\vertex (a25) at (0.5,-0.5) {};		
	
	\vertex (a31) at (0.15,0.85) {};
	\vertex (a32) at (-0.12,0.85) {};	
	\vertex (a33) at (0.45,0.8) {};
	\vertex (a34) at (0.25,0.7) {};
	\vertex (a35) at (-0.35,0.7) {};

	\draw[thick] (a31) -- (a32) -- (a35) -- (a34) -- (a33) -- (a31);

	\draw[thick] (a21) -- (a22) -- (a23) -- (a24) -- (a25) -- (a21);

	\draw[thick] (a11) -- (a12) -- (a13) -- (a14) -- (a15) -- (a11);

	\draw[thick] (a22) -- (a33);
	\draw[thick] (a23) -- (a34);

	\draw[thick] (a15) -- (a25);
	\draw[thick] (a14) -- (a24);
	\draw[thick] (a13) -- (a35);

\end{scope}
\end{tikzpicture}
\end{center}
    \caption{The reduction of $P_\nu$ in the length reduction order for $\nu = (NE)^3$ encodes a sequence of edge truncations of an $3$-simplex giving rise to the $3$-cyclohedron}
    \label{fig:truncationsCyclo}
\end{figure}
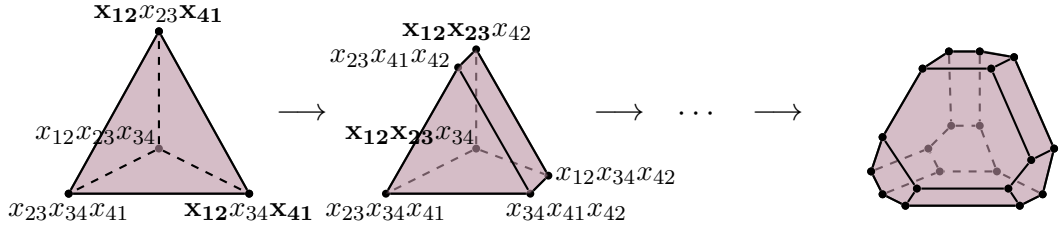

\begin{remark}
When $\nu$ is the staircase path $(NE)^{n}$, reductions of $P_\nu$ can be viewed as successive edge truncations of an $n$-simplex. 
The truncation order induced by $\rho_{\text{len}}$ yields the cyclohedron, as in Figure~\ref{fig:truncationsCyclo}. 
\end{remark}

\section*{Acknowledgements}

The author thanks Martha Yip for many inspiring conversations and helpful suggestions.


\DeclareRobustCommand{\VAN}[3]{#3}
\printbibliography

\end{document}